\documentclass{amsart}
\usepackage{a4,amssymb,amsmath}

\numberwithin{equation}{section}

\newtheorem{thm}{Theorem}[section]
\newtheorem{lem}[thm]{Lemma}
\newtheorem{prop}[thm]{Proposition}

\theoremstyle{definition}
 \newtheorem{defi}[thm]{Definition}

\theoremstyle{remark}

\newcommand{\R}{\mathbb{R}}
\newcommand{\N}{\mathbb{N}}
\newcommand{\hn}{\mathbb{H}^n}
\newcommand{\Cn}{\mathbb{C}^n}
\newcommand{\Leb}{\mathcal{L}^{2n+1}}
\newcommand{\Rn}{\mathbb{R}^n}

\newcommand{\p}{\mathcal P}

\newcommand{\spt}{\operatorname{spt}}
\newcommand{\card}{\operatorname{card}}
\newcommand{\diam}{\operatorname{diam}}

\newcommand{\leb}{\operatorname{Leb}}

\newcommand{\im}{\operatorname{Im}}
\newcommand{\1}{{\rm 1\mskip-4mu l}}
\newcommand{\spa}{\operatorname{span}}

\newcommand{\e}{\varepsilon}
\newcommand{\g}{\gamma}
\newcommand{\s}{\sigma}
\newcommand{\G}{\Gamma}
\newcommand{\Om}{\Omega}

\begin{document}

\title{Monge's transport problem in the Heisenberg group}

\author[L. De Pascale]{L. De Pascale}
\address[L. De Pascale]{Universit\'a di Pisa, Dipartimento di Matematica
  Applicata, Via Buonarroti 1/c, 56127
  Pisa, Italy}

\author[S. Rigot]{S. Rigot}
\address[S. Rigot]{Universit\'e de Nice Sophia-Antipolis, Laboratoire J.-A. Dieudonn\'e, CNRS-UMR 6621, Parc Valrose, 06108 Nice cedex 02, France}

\keywords{Monge transport problem, Optimal transport map, Monge-Kan\-torovich problem, Heisenberg group, Sub-Riemannian distance}

\subjclass[2000]{49Q20 (53C17)}

\begin{abstract}
We prove the existence of solutions to Monge's transport problem between two compactly supported Borel probability measures in the Heisenberg group equipped with its Carnot-Carath\'eodory distance assuming that the initial measure is absolutely continuous with respect to the Haar measure of the group. 
 
\end{abstract}

\maketitle

%%%%
%%%%%%%%%%%%%%%%%%%%%%%
\section{Introduction}

The classical Monge's transport problem refers to the problem of moving one distribution of mass onto another as efficiently as possible, where the efficiency criterion is expressed in terms of the average distance transported. It originates from a paper by G. Monge, \textit{M\'emoire sur la th\'eorie des d\'eblais et des remblais}, in 1781. Rephrased and generalized in modern mathematical terms, we are given two Borel probability measures $\mu$ and $\nu$ on a metric space $(X,d)$ and we want to minimize
\begin{equation*}
 T \mapsto \int_X d(x,T(x))\,d\mu(x)
\end{equation*}
among all transport maps $T$ from $\mu$ to $\nu$, i.e., all $\mu$-measurable maps $T:X\to X$ such that $T_\#\mu=\nu$, meaning that $\nu(B)=\mu(T^{-1}(B))$ for all Borel set $B$. 

In this paper, we are interested in Monge's transport problem in the Heisenberg group $(\hn,d)$ equipped with its Carnot-Carath\'eodory distance. We prove the existence of an optimal transport map between two compactly supported Borel probability measures $\mu$ and $\nu$ on $\hn$ assuming that the first measure $\mu$ is absolutely continuous with respect to the Haar measure $\Leb$ of $\hn$. 

\begin{thm} \label{mainthm}
Let $\mu$ and $\nu$ be two compactly supported Borel probability measures on $\hn$. Assume that $\mu \ll \Leb$. Then there exists an optimal transport map solution to Monge's transport problem between $\mu$ and $\nu$, i.e., a $\mu$-measurable map $T:\hn\rightarrow\hn$ such that $T_\#\mu=\nu$ and
\begin{equation*}
\int_{\hn} d(x,T(x))\,d\mu(x) = \inf_{S_\#\mu=\nu} \int_{\hn} d(x,S(x))\,d\mu(x).
\end{equation*}
\end{thm}

Monge's transport problem in $\Rn$ equipped with a distance induced by a norm has already been widely investigated. A first attempt to solve this problem goes back to the work of Sudakov \cite{sudakov}. It was however discovered some years later that the proof in \cite{sudakov} was not completely correct. In \cite{evans-gangbo} a PDE-based alternative to Sudakov's approach has been developed. The authors prove the existence of an optimal transport map in $\Rn$ equipped with the Euclidean norm under the assumptions that $\spt \mu \cap \spt \nu = \emptyset$, $\mu$, $\nu \ll \mathcal{L}^{n}$ with Lipschitz densities with compact support. Existence results for general absolutely continuous measures $\mu$, $\nu$ with compact support have been obtained independently in \cite{cfm} and \cite{TrudWang} and have been extended to a Riemannian setting in \cite{FeldMcC}. The existence of a solution to Monge's transport problem assuming only that the initial measure $\mu$ is absolutely continuous has been proved in \cite{ambrosio}, see also \cite{ap}, \cite{akp}, \cite{caravenna1}. All these later proofs roughly involve a Sudakov-type dimension reduction argument, via different technical implementations though, and require some regularity assumptions about the norm $\Rn$ is endowed with. For some time, it seemed that there were indeed some borderline cases about the norms that could not be attacked through these techniques. 

Recently another approach that does not go through Sudakov-type arguments and in particular does not require disintegration of measures has been developed in \cite{champion-dePascale}, see also \cite{champion-dePascale-first}, to solve Monge's transport problem for general norms in $\Rn$. This approach relies on rather simple but powerful density results. In the present paper we follow closely this approach. We basically show that a very similar strategy can be implemented in the context of the Heisenberg group equipped with its Carnot-Carath\'eodory distance. The main features that play a role in this approach are that $(\hn, d, \Leb)$ is a doubling polish metric measure space, a non-branching geodesic space and satisfies a so-called Measure Contraction Property. It is very likely that this approach can be extended to more general metric measure spaces, see Section~\ref{extensions}. We have chosen however to present the particular case of the Heisenberg group for simplicity, this space being moreover an instructive explicit example of non Riemannian space.~\footnote{The preprint~\cite{bianc-cav} which appeared during the completion of the present paper also addresses Monge's transport problem in metric spaces with a geodesic distance using a Sudakov-type dimension reduction argument and disintegration of measures.} 

The strategy starts by considering the nowadays classical relaxation of Monge's transport problem proposed by Kantorovich. In Kantorovich' formulation one considers transport plans, i.e., Borel probability measures on $X \times X$ with first and second marginals $\mu$ and $\nu$ respectively. Denoting by $\Pi(\mu,\nu)$ the class of all transport plans, one wants to minimize
\begin{equation*}
 \g \mapsto \int_{X\times X} d(x,y)\,d\g(x,y)
\end{equation*}
among all $\g \in \Pi(\mu,\nu)$. Due to the linearity of the constraint $\g \in \Pi(\mu,\nu)$, weak topologies provide existence of optimal transport plans. As a classical fact it turns out that whenever an optimal transport plan is induced by a $\mu$-measurable map $T$, i.e., can be written in the form $(I \otimes T)_\sharp\mu$ where $(I \otimes T)(x) := (x,T(x))$, then $T$ is an optimal transport map between $\mu$ and $\nu$ solution to Monge's transport problem. We follow here this scheme seeking after optimal transport plans that will be shown to be eventually induced by $\mu$-measurable maps.

In our present context we first prove that any optimal transport plan is concentrated on a set of pairs of points that are connected by a unique minimal curve and that these transport rays cannot bifurcate (see Section~\ref{sect:optplanning}). Next, following ideas already introduced in the literature and more specifically here inspired by~\cite{sant}, one introduces variational approximations (see Section~\ref{varapprox}). This procedure allows to select optimal transport plans with specific properties. These transport plans will eventually be proved to be induced by $\mu$-measurable maps. This procedure is here essentially twofold. On one hand it allows to select optimal transport plans that are solution to a secondary variational problem. This secondary variational problem prescribes the geometry of transport rays. The selected transport plans are indeed shown to be monotonic along transport rays (see Lemma~\ref{pi2}). On the other hand, given a transport plan, one can in our context interpolate between its first and second marginal in a natural way (see Subsection~\ref{sect-interpolation}). Absolute continuity and more importantly $L^\infty$-estimates on the density of the interpolations will play an important role and one can indeed prove $L^\infty$-estimates on the interpolations in the approximating variational problems (see Proposition~\ref{e:Linftydensityestimates}). These estimates rely on the so-called Measure Contraction  Property of $\hn$. In the limit one will  eventually get suitable $L^\infty$-estimates on the interpolations constructed from optimal transport plans selected through the variational approximation procedure. Next we note that some properties of plans with absolutely continuous first marginal proved in~\cite{champion-dePascale} can be easily generalized to our setting (see Section~\ref{lebpoints}). These properties are independent of the transport problem. They rely on the notion of Lebesgue points of functions and Lebesgue points of sets, notions which make sense for instance in any doubling metric measure space. Together with the above mentioned $L^\infty$-estimates on the interpolations, one can in particular prove density estimate on the transport set of selected optimal transport plans in the same way as in~\cite{champion-dePascale} (see Section~\ref{main}). All together, namely combining this later density estimate on the transport set (Lemma~\ref{mainlemma}) with Lemma~\ref{dens2} and remembering the monotonicity along transport rays (Lemma~\ref{pi2}) it turns out that the selected transport plans are necessarily induced by a transport map as eventually proved in Theorem~\ref{mainthmbis}.

The paper is organized as follows. In Section~\ref{opttrans} we recall classical facts about optimal transportation for later use. In Section~\ref{prelimhn} we describe the Heisenberg group focusing on the features that will be needed in this paper. In Section~\ref{sect:optplanning} we prove geometric properties of optimal transport plans and the monotonicity along transport rays of solutions to  the secondary variational problem. The variational approximations are introduced and studied in Section~\ref{varapprox}. In Section~\ref{lebpoints} we state in our framework properties of plans with absolutely continuous first marginal proved in~\cite{champion-dePascale} in $\Rn$. This section, independent of the transport problem, contains density results that play an essential role in the strategy followed here. In Section~\ref{main} we prove lower bounds on the density, in some suitable sense, of the transport set of optimal transport plans selected through the variational approximations. We conclude in Section~\ref{conclusion} proving that the selected transport plans are induced by a transport map. We discuss in the final Section~\ref{extensions} some possible extensions of this approach to other spaces.

%%%%
%%%%%%%%%%%%%%%%%%%%%%%
\section{Preliminaries on optimal transportation} \label{opttrans}

We recall some well-known facts about optimal transportation confining ourselves to statements that will fit our needs in the rest of the paper. More general versions of these results hold in more general contexts. We refer to e.g. \cite{villani} and the references therein.

Let $(X,d)$ be a Polish space, i.e., a complete and separable metric space. We denote by $\p(X)$ the set of all Borel probability measures on $X$ and by $\p_c(X)$ the set of Borel probability measures on $X$ with compact support. The weak topology we consider on $\p(X)$ is the topology induced by convergence against bounded and continuous test functions (or narrow topology).

\subsection{Kantorovich transport problem} 
Let $\mu$, $\nu\in\p(X)$. We denote by 
\begin{equation*}
 \Pi(\mu,\nu) := \{ \g \in \p(X \times X);\, (\pi_1)_\sharp\g = \mu,\, (\pi_2)_\sharp\g = \nu \}
\end{equation*}
the set of all transport plans between $\mu$ and $\nu$. Here $\pi_1$, $\pi_2: X\times X \rightarrow X$ denote the canonical projections on the first and second factor respectively.

Given $c:X\times X \rightarrow [0,+\infty]$ a lower semicontinuous cost function, we look at Kantorovich transport problem between $\mu$ and $\nu$ with cost $c$:
\begin{equation} \label{e:MK} %\tag{MK} 
 \min_{\g \in \Pi(\mu,\nu)} \int_{X\times X} c(x,y)\,d\g(x,y).
\end{equation}
As a classical fact, existence of solutions to \eqref{e:MK} follows from the weak compactness of $\Pi(\mu,\nu)$ together with the lower semicontinuity of the functional to be minimized. We call them optimal transport plans.

\subsection*{Cyclical monotonicity} We say that a set $\G\subset X\times X$ is $c$-cyclically monotone if 
\begin{equation*}
\sum_{i=1}^N c(x_i,y_i) \leq \sum_{i=1}^N c(x_{i+1},y_i)
\end{equation*}
whenever $N\geq 2$ and $(x_1,y_1), \dots, (x_N,y_N)\in\G$.

\begin{thm} \label{ccycl}
Let $\g \in \Pi(\mu,\nu) $ be an optimal transport plan and assume that $\int_{X\times Y} c(x,y)\,d\g <+\infty$. Then $\g$ is concentrated on a $c$-cyclically monotone Borel set.
\end{thm}

\subsection*{Dual formulation - Kantorovich potentials}
Let $\psi:X\rightarrow \R \cup \{-\infty\}$. We say that $\psi$ is $c$-concave if $\psi\not\equiv -\infty$ and if there exists $\varphi:X \rightarrow \R \cup \{-\infty\}$, $\varphi\not\equiv -\infty$, such that 
\begin{equation*}
\psi(x) = \inf_{y\in X} c(x,y) - \varphi(y).
\end{equation*}

\begin{thm}  \label{duality}
In addition to the previous assumptions, assume that $c$ is real-valued and that 
\begin{equation*}
 \forall\, (x,y)\in X\times X, \qquad c(x,y) \leq a(x) + b(y)
\end{equation*}
for some $a\in L^1(\mu)$ and $b\in L^1(\nu)$. Then one has
\begin{equation} \label{e:duality}
 \min_{\g\in \Pi(\mu,\nu)} \int_{X\times X} c(x,y)\,d\g(x,y) 
= \max \int_{X} \psi(x)\,d\mu(x) + \int_{X} \psi^c(y)\,d\nu(y)
\end{equation}
where the above maximum is taken among all $c$-concave functions $\psi$ and $ \psi^c(y) := \inf_{x\in X} c(x,y) - \psi(x)$. 
\end{thm}

\begin{defi} [Kantorovich potentials]
 We say that $\psi:X\rightarrow \R \cup \{-\infty\}$ is a Kantorovich potential if $\psi$ is a $c$-concave maximizer for the right-hand side of \eqref{e:duality}.
\end{defi}

\begin{thm} \label{caracterisation_opt_planning}
With the same assumptions as in Theorem~\ref{duality}, let $\psi$ be a Kantorovich potential. Then $\g\in \Pi(\mu,\nu)$ is an optimal transport plan if and only if 
\begin{equation*}
 c(x,y) = \psi(x) + \psi^c(y) \qquad \g-\text{a.e. in } X\times X.
\end{equation*}
\end{thm}

We will use these results for various cost functions. In the particular case which is the core of this paper and where $c(x,y) = d(x,y)$ and $\mu$, $\nu \in \p_c(X)$, one can rephrase these results in terms of 1-Lipschitz Kantorovich potentials. More precisely, set 
\begin{equation*}
 \text{Lip}_1(d) := \{u:X\rightarrow \R;\quad |u(x) - u(y)| \leq d(x,y) \quad \forall\, x, y \in X\}.
\end{equation*}

\begin{thm} \label{1lip_potential}
Let $\mu$, $\nu \in \p_c(X)$. Then one can find a Kantorovich potential $u\in \text{Lip}_1(d) $ so that
\begin{equation*}
 \min_{\g\in \Pi(\mu,\nu)} \int_{X\times X} d(x,y)\,d\g(x,y) 
= \int_{X} u(x)\,d\mu(x) - \int_{X} u(y)\,d\nu(y)
\end{equation*}
and 
$\g\in \Pi(\mu,\nu)$ is an optimal transport plan solution to Kantorovich transport problem~\eqref{e:MK} between $\mu$ and $\nu$ with cost $c(x,y) = d(x,y)$ if and only if 
\begin{equation*}
u(x) - u(y) = d(x,y)  \qquad \g-\text{a.e. in } X\times X.
\end{equation*}
\end{thm}

\subsection{Transport problem} Let $\mu$, $\nu\in\p(X)$. We say that a $\mu$-measurable map $T:X\rightarrow X$ is a transport map between $\mu$ and $\nu$ if $T_\sharp\mu=\nu$, i.e., $\nu(B)=\mu(T^{-1}(B))$ for
all Borel set $B$. 

Given $c:X\times X \rightarrow [0,+\infty[$ a continuous cost function, we look at the transport problem between $\mu$ and $\nu$ with cost $c$:
\begin{equation}  \label{e:M} %\tag{M}
 \min_{T_\#\mu=\nu} \int_{X} c(x,T(x))\,d\mu(x).
\end{equation}

We say that a transport plan $\g\in \Pi(\mu,\nu)$ is induced by a transport if there exists a $\mu$-measurable map $T:X\rightarrow X$ such that $(I \otimes T)_\sharp\mu = \g$ where $(I \otimes T)(x) := (x,T(x))$. Such a map is automatically a transport map between $\mu$ and $\nu$. We also recall that if a transport plan $\gamma$ is concentrated on a $\gamma$-measurable graph then $\gamma$ is induced by a transport.

\begin{thm} [Optimal transport plans versus optimal transport maps] \hfill \label{plan-transport}

(i) Assume that $\g$ is an optimal transport plan solution to Kantorovich transport problem~\eqref{e:MK} and that $\g$ is induced by transport $T$. Then $T$ is an optimal transport map solution to the transport problem~\eqref{e:M}.

(ii) Assume that any optimal transport plan solution to Kantorovich transport problem~\eqref{e:MK} is induced by transport. Then there exists a unique optimal transport map solution to the transport problem~\eqref{e:M}.
 \end{thm}

%%%%
%%%%%%%%%%%%%%%%%%%%%%%
\section{Preliminaries on $\hn$} \label{prelimhn}

We consider the Heisenberg group $\hn$ equipped with its Carnot-Carath\'eodory distance. Endowed with this distance $\hn$ is a polish geodesic and non-branching metric space and a doubling metric measure space when equipped with its Haar measure.

\subsection{The Heisenberg group}
The Heisenberg group $\hn$ is a connected, simply connected Lie group with stratified Lie algebra. We identify it with $\Cn \times \R$ equipped with the group law 
\begin{equation*}
[\zeta,t] \cdot [\zeta',t'] := [\zeta + \zeta', t+t'+ 2 \sum_{j=1}^n \im \zeta_j \overline\zeta'_j]
\end{equation*}
where $\zeta = (\zeta_1,\dots,\zeta_n)$, $\zeta' = (\zeta'_1,\dots,\zeta'_n)\in\Cn$ and $t$, $t'\in \R$. The unit element is 0 and the center of the group is 
\begin{equation*}
L := \{[0,t] \in \hn ;\; t \in \R\}.
\end{equation*}
There is a natural family of dilations $\delta_r$ on $\hn$ defined by $\delta_r([\zeta,t]):= [r\zeta,r^2 t]$. These dilations are group homomorphisms. 

We may also identify $\hn$ with $\R^{2n+1}$ via the correspondence $[\zeta,t]= (\xi,\eta,t)$ where $\xi=(\xi_1,\dots,\xi_n)$, 
$\eta=(\eta_1,\dots,\eta_n) \in \R^n$, $t\in \R$ and $\zeta = (\zeta_1,\dots,\zeta_n)\in\Cn$ with $\zeta_j = \xi_j + i \eta_j$. The horizontal subbundle of the tangent bundle is defined by
\begin{equation*}
\mathcal{H} := \spa \left\{X_j; \, j=1,\dots,n\right\} \oplus \spa \left\{Y_j; \, j=1,\dots,n\right\}
\end{equation*}
where the left invariant vector fields $X_j$ and $Y_j$ are given by
\begin{equation*}
X_j := \partial_{\xi_j} + 2 \eta_j \partial_t \,, \quad Y_j := \partial_{\eta_j} - 2 \xi_j \partial_t.
\end{equation*}
Vector fields in $\mathcal{H}$ will be called horizontal vector fields. Setting $T:=\partial_t$, the only non trivial bracket relations are $[X_j,Y_j] = -4 T$ hence $\spa \{T\} = [ \mathcal{H},\mathcal{H}]$ and the Lie algebra $\mathcal{H}^n$ of $\hn$ admits the stratification $\mathcal{H}^n = \mathcal{H} \oplus \spa \{T\}$.

The Lebesgue measure $\Leb$ on $\hn \thickapprox \mathbb{R}^{2n+1}$ is a Haar measure of the group. It is $(2n+2)$-homogeneous with respect to the dilations, 
\begin{equation*}
 \Leb(\delta_r(A)) = r^{2n+2} \Leb(A)
\end{equation*}
for all Borel set $A$ and all $r>0$.

\subsection{Carnot-Carath\'eodory distance} The Carnot-Carath\'eodory distance on $\hn$ is defined by 
\begin{equation} \label{e:cc-dist}
 d(x,y) = \inf \{ length_{g_0} (\gamma); \; \gamma \text{ horizontal } C^1 \text{-smooth curve joining } x \text{ to } y\},
\end{equation}
where a $C^1$-smooth curve is said to be horizontal if, at every point, its tangent vector belongs to the horizontal subbundle of the tangent bundle and $g_0$ is the left invariant Riemannian metric which makes $(X_1,\dots,X_n,Y_1,\dots,Y_n,T)$ an orthonormal basis. For a general presentation of Carnot-Carath\'eodory spaces, see e.g. \cite{bel}, \cite{mont}. 

The topology induced by this distance is the original (Euclidean) topology on $\hn \thickapprox (\mathbb{R}^{2n+1},g_0)$ and $(\hn,d)$ is a  complete metric space. The distance is left invariant and 1-homogeneous with respect to the dilations,
\begin{equation*}
d(x \cdot y, x\cdot z) = d(y,z) \quad \text{and} \quad d(\delta_r(y), \delta_r(z)) = r\,d(y,z)
\end{equation*}
for all $x$, $y$, $z\in\hn$ and all $r>0$. It follows in particular that $B(x,r) = x\cdot \delta_r (B(0,1))$ and hence 
\begin{equation} \label{e:measball}
 \Leb(B(x,r)) = c_n \, r^{2n+2}
\end{equation}
for all $x\in \hn$, all $r>0$ and where $c_n := \Leb(B(0,1))>0$. The measure $\Leb$ is in particular a doubling measure on $(\hn,d)$. For more details about doubling metric measure spaces, see e.g. \cite{heinonen}.

Endowed with its Carnot-Carath\'eodory distance $\hn$ is a geodesic space, i.e., for all $x$, $y\in\hn$, there exists a curve $\s \in C([a,b],\hn)$ such that $\s(a)=x$, $\s(b)=y$ and $d(x,y) = l(\s)$ where 
\begin{equation*}
 l(\s) = \sup_{N\in  \N^*} \sup_{a=t_0\leq \dots \leq t_N = b} \sum_{i=0}^{N-1} d(\s(t_i),\s(t_{i+1})).
\end{equation*}
Up to a reparameterization one can always assume that length minimizing curves $ \s$ are parameterized proportionally to arc-length, i.e., 
\begin{equation*}
d(\s(s),\s(s')) = v \,(s'-s)
\end{equation*}
for all $s<s' \in [a,b]$, where $v:=d(\s(a),\s(b))/(b-a)$ is the (constant) speed of the curve. As a convention we will use throughout this paper the terminology \textit{minimal curves} to denote length minimizing curves parameterized proportionally to arc-length.

\begin{defi}[Minimal curves] 
We say that a continuous curve $\s:[a,b]\rightarrow\hn$ is a minimal curve if $l(\s) = d(\s(a),\s(b))$ and $\s$ is parameterized proportionally to arc-length.
 \end{defi}

In general Carnot-Carath\'eodory spaces, issues about uniqueness and regularity of minimal curves between any two points as well as issues about the regularity of the distance function to a given point could be delicate. In the specific case of the Heisenberg group, equations of all minimal curves can be explicitly computed and exploited to overcome these difficulties. We recall below the description of minimal curves in $\hn$, see e.g. \cite{gaveau}, \cite{ar}. We set
\begin{equation} \label{e:omega}
 \Om := \{(x,y)\in \hn\times \hn;\,\, x^{-1}\cdot y \not \in L\}.
\end{equation}

\begin{thm} [Minimal curves in $\hn$] \label{geod} 
Minimal curves in $\hn$ are horizontal $C^1$-smooth curves such that the infimum in \eqref{e:cc-dist} is achieved. One has the more precise description:

\renewcommand{\theenumi}{\roman{enumi}}
\begin{enumerate}

\item Non trivial minimal curves starting from 0 and parameterized on $[0,1]$ are all curves $\s_{\chi,\varphi}$ for some $\chi \in \Cn\setminus\{0\}$ and $\varphi\in [-2\pi,2\pi]$ where 
\begin{equation*}
\s_{\chi,\varphi}(s) = [i \dfrac{(e^{-i\varphi s} -1) \chi}{\varphi}, 2 |\chi|^2 \,\dfrac{\varphi s - \sin (\varphi s)}{\varphi^2}]
 \end{equation*}
if $\varphi\in [-2\pi,2\pi]\backslash \{0\}$ and 
\begin{equation*}
 \s_{\chi,\varphi}(s) = [\chi s, 0]
\end{equation*}
if $\varphi =0$. Moreover one has $|\chi| = d(0,\s_{\chi,\varphi}(1))$.

\item For all $(x,y)\in \Om$, there is a unique minimal curve $x\cdot \s_{\chi,\varphi}$ between $x$ and $y$ for some $\chi \in \Cn\setminus\{0\}$ and some $\varphi\in\, (-2\pi,2\pi)$ and one has $|\chi| = d(x,y)$.

\item 
If $(x,y)\not\in \Om$, $x^{-1}\cdot y = [0,t]$ for some $t\in \R^*$, there are infinitely many minimal curves between $x$ and $y$. These curves are all curves of the form $x \cdot \s_{\chi,2\pi}$ if $t>0$, $x \cdot \s_{\chi,-2\pi}$ if $t<0$, for all $\chi \in \Cn$ such that $|\chi| = \sqrt{\pi |t|}$.
\end{enumerate}
\end{thm}

Here and in the following, $|\chi| = (\sum_{j=1}^n |\chi_j|^2)^{1/2}$ for $\chi = (\chi_1,\dots,\chi_n)\in\Cn$. In particular it follows from this description that $(\hn,d)$ is non-branching.

\begin{prop} [Non-branching property of $\hn$] \label{nonbranching}
 The space $(\hn,d)$ is non-branching, i.e., any two minimal curves which coincide on a non trivial interval coincide on the whole intersection of their intervals of definition. 
\end{prop}

Equivalently for any quadruple of points $z$, $x$, $y$, $y'\in \hn$, if $z$ is a midpoint of $x$ and $y$ as well as a midpoint of $x$ and $y'$, then $y=y'$. 

The next lemma collects some differentiability properties of the distance function to a given point to be used later. For $y\in \hn$, we set $L_y:= y\cdot L$.

\begin{lem} \label{prop-distcc} Let $y\in \hn$ and set $d_y(x) := d(x,y)$. Then the function $d_y$ is of class $C^{\infty}$ on $\hn \backslash L_y$ (equipped with the usual differential structure when identifying $\hn$ with $\mathbb{R}^{2n+1}$). Moreover one has 

\renewcommand{\theenumi}{\roman{enumi}}
\begin{enumerate}

\item $|\nabla_H d_y(x)| = 1$ for all $x \in \hn \backslash L_y$ where $$\nabla_H d_y(x) := (X_1 d_y(x)+i Y_1 d_y(x), \dots, X_n d_y(x)+i Y_n d_y(x)).$$ 

\item If $\nabla d_y(x) =  \nabla d_{y'}(x)$ and $d(x,y) = d(x,y')$ for some $x \in \hn \backslash ( L_y\cup L_{y'})$, then $y=y'$. Here $\nabla = (\partial_{\xi_1},\cdots,\partial_{\xi_n},\partial_{\eta_1},\cdots,\partial_{\eta_n}, \partial_t)$ denotes the classical gradient when identifying $\hn$ with $\mathbb{R}^{2n+1}$. 
 \end{enumerate}
\end{lem}

\begin{proof} Set $\Phi(\chi,\varphi):=\s_{\chi,\varphi}(1)$ where $\s_{\chi,\varphi}$ is given in Theorem~\ref{geod}. This map is a $C^\infty$-diffeomorphism from $\Cn\setminus\{0\} \times (-2\pi,2\pi)$ onto $\hn\setminus L$ (see e.g. \cite{monti}, \cite{ar}, \cite{juillet}). If $x = \Phi(\chi,\varphi) \in \hn\setminus L$ with $(\chi,\varphi) \in \Cn\setminus\{0\} \times (-2\pi,2\pi)$, one has $d_0(x) = |\chi|$ and 
\begin{equation*}
\nabla_H d_0(x) = \dfrac{\chi}{|\chi|} e^{-i\varphi} \quad \text{and} \quad \partial_t d_0(x) = \dfrac{\varphi}{4|\chi|},
\end{equation*}
see \cite[Lemma 3.11]{ar}. Next, by left invariance, we have $d_y(x) = d_0(y^{-1}\cdot x)$, $\nabla_H d_y(x) = \nabla_H d_0(y^{-1}\cdot x)$ and $\partial_t d_y(x) = \partial_t d_0(y^{-1}\cdot x)$ if $x\in \hn \backslash L_y$ and the lemma follows easily.
\end{proof}

\subsection{Interpolation between measures} \label{sect-interpolation} The notion of interpolation constructed from a transport plan between any two measures will be one of the key notion to be used later. To define it in our geometrical context, we first fix a measurable selection of minimal curves, i.e., a Borel map $S:\hn\times\hn \rightarrow C([0,1], \hn)$ such that for all $x$, $y\in\hn$, $S(x,y)$ is a minimal curve joining $x$ and $y$. The existence of such a measurable recipe to join any two points in $\hn$ by a minimal curve follows from general theorems about measurable selections, see e.g. \cite[Chapter 7]{villani}. Next we set $e_t(\s) := \s(t)$ for all $\s \in C([0,1],\hn)$ and $t\in [0,1]$. In particular $e_t(S(x,y))$ denotes the point lying at distance $t \, d(x,y)$ from $x$ on the selected minimal curve $S(x,y)$ between $x$ and $y$. 

\begin{defi} \label{interpolation} 
 Let $\mu$, $\nu\in \p(\hn)$ and let $\g \in \Pi(\mu,\nu)$. The interpolations between $\mu$ and $\nu$ constructed from $\g$ are defined as the family $((e_t \circ S)_\sharp \g))_{t\in [0,1]}$ of Borel probability measures on $\hn$. 
\end{defi}

Note that these interpolations depend a priori on the measurable selection $S$ of minimal curves. This is actually not a serious issue for our purposes. We will moreover always consider interpolations constructed from transport plans that are concentrated on the set $\Om$ on which $S(x,y)$ is nothing but the unique minimal curve between $x$ and $y$. Note also for further reference that $S\lfloor_\Om$ is continuous.

\subsection{Intrinsic differentiability} Intrinsic differentiability properties of real-valued Lipschitz functions on $\hn$, namely a Rademacher's type theorem, will be useful when considering 1-Lipschitz Kantorovich potentials. This theorem is a particular case of a more general result due to P. Pansu. We say that a group homomorphism $g :\hn \rightarrow \R$ is homogeneous if $g(\delta_r(x)) = r\,g(x)$ for 
all $x\in\hn$ and all $r >0$. 

\begin{defi}
 We say that a map $f:\hn \rightarrow \R$ is Pansu-differentiable at $x \in \hn$ if there exists 
an homogeneous group homomorphism $g :\hn \rightarrow \R$ such that
\begin{equation*}
\lim_{y \rightarrow x} \frac{f(y) - f(x) - g(x^{-1} \cdot y)}{d(y,x)} = 0.
\end{equation*}
The map $g$ is then unique and will be denoted by $D_H f(x)$.
\end{defi}

If $f:\hn \rightarrow \R$ is Pansu-differentiable at $x \in \hn$ then the maps $s \mapsto f(x \cdot \delta_s[e_j, 0])$, resp. $s \mapsto f(x\cdot \delta_s[e_{n+j}, 0])$, are differentiable at $s=0$ and if we denote the corresponding derivatives by $X_jf(x)$, resp. $Y_jf(x)$, then 
\begin{equation*}
D_H f(x)(\xi,\eta,t) = \sum_{j=1}^n \xi_j X_jf(x) + \eta_j Y_jf(x).
\end{equation*}
Here $e_j = (\delta_{1}^{j},\dots,\delta_{n}^{j})\in \Cn$ and $e_{n+j} = (i\delta_{1}^{j},\dots,i\delta_{n}^{j})\in \Cn$. Using similar notations as in the classical smooth case, we then set $\nabla_H f(x) := (X_1 f(x)+i Y_1 f(x), \dots, X_n f(x)+i Y_n f(x))$.

\begin{thm}[Pansu-differentiability theorem] \label{PansuRademacher} \cite{pansu}
Let $f:(\hn,d) \rightarrow \R$ be a $C$-Lipschitz function. Then, for $\Leb$-a.e. $x\in\hn$, the function $f$ is Pansu-differentiable at $x$ and $|\nabla_H f(x)|\leq C$.
\end{thm}

The next lemma will be used to prove that any optimal transport plan is concentrated on the set $\Om$.

\begin{lem} \label{uniquegeod}
 Let $u\in \text{Lip}_1(d)$, $x\in\hn$ be such that $u$ is Pansu-differentiable at $x$ with $|\nabla_H u(x)| \leq 1$ and let $y\in\hn$ be such that $u(x) - u(y) = d(x,y)$. Then there exists a unique minimal curve between $x$ and $y$.
\end{lem}

\begin{proof}
 Let $\s:[0,1]\rightarrow\hn$ be a minimal curve between $x$ and $y$. Then $\s$ is a horizontal $C^1$-smooth curve and if $\s(t) = (\s_1(t),\dots,\s_{2n+1}(t)) \in \hn \thickapprox \R^{2n+1}$, one has for all $t\in [0,1]$,
\begin{equation*}
 \dot\s(t) = \sum_{j=1}^n \dot\s_j(t)\, X_j(\s(t)) + \dot\s_{n+j}(t)\, Y_j(\s(t))
\end{equation*}
and $|\dot\s_H(t)| = d(x,y)$ where $\dot\s_H(t) := (\dot\s_1(t) +i\,\dot\s_{n+1}(t),\dots,\dot\s_n(t) +i\,\dot\s_{2n}(t))\in\Cn$. On the other hand, one has 
\begin{equation*}
 u(x) - u(\s(t)) = d(x,\s(t)) = t\, d(x,y)
\end{equation*}
for all $t\in [0,1]$. Differentiating this equality with respect to $t$, we get
\begin{equation*}
 \sum _{j=1}^n \dot\s_j(0)\, X_ju(x) + \dot\s_{n+j}(0) \, Y_ju(x)= \dfrac{d}{dt}\,u(\s(t))_{|_{t=0}} = - d(x,y).
\end{equation*}
All together, it follows that 
\begin{equation*}
d(x,y) = |\sum _{j=1}^n \dot\s_j(0)\, X_ju(x) + \dot\s_{n+j}(0) \, Y_ju(x)| \leq  |\nabla_H u(x)|\, |\dot\s_H(0)| \leq d(x,y).
\end{equation*}
In particular, there is equality in all the previous inequalities which implies in turn that $\dot\s_H(0) = - \,d(x,y) \,\nabla_H u (x)$. On the other hand one knows from Theorem~\ref{geod} that $\s = x\cdot \s_{\chi,\varphi}$ for some $\chi \in \Cn\setminus\{0\}$ and $\varphi\in\, [-2\pi,2\pi]$. In particular one has $\dot\s_H(0) = \chi$. It follows that $\chi = - \,d(x,y) \,\nabla_H u (x)$ is uniquely determined hence there is a unique minimal curve joining $x$ and $y$ according once again to the description given in Theorem~\ref{geod}.
\end{proof}

%%%%
%%%%%%%%%%%%%%%%%%%%%
\section{Properties of $\Pi_1(\mu,\nu)$ and $\Pi_2(\mu,\nu)$} \label{sect:optplanning}

Let $\mu$, $\nu \in \p_c(\hn)$ be fixed. We denote by $\Pi_1(\mu,\nu)$ the set of optimal transport plans solution to Kantorovich transport problem~\eqref{e:MK} between $\mu$ and $\nu$ with cost $c(x,y) = d(x,y)$. 

We first prove some geometric properties of optimal transport plans. These properties follow from the behavior of minimal curves in $(\hn,d)$. In the next lemma, we prove that any optimal transport plan is concentrated on the set $\Om$ (see \eqref{e:omega}) of pair of points that are connected by a unique minimal curve.

\begin{lem} \label{pi1.1}
 Let $\g\in \Pi_1(\mu,\nu)$ and assume that $\mu\ll\Leb$. Then for $\g$-a.e. $(x,y)$, there exists a unique minimal curve between $x$ and $y$.
\end{lem}

\begin{proof}
Let $u\in \text{Lip}_1(d)$ be a Kantorovich potential associated to Kantorovich transport problem~\eqref{e:MK} between $\mu$ and $\nu$ with cost $c(x,y) = d(x,y)$ (see Section~\ref{opttrans} and Theorem~\ref{1lip_potential} there). Since $u\in \text{Lip}_1(d)$, we know from Theorem \ref{PansuRademacher} that for $\Leb$-a.e., and  hence $\mu$-a.e., $x\in\hn$, $u$ is Pansu-differentiable at $x$ with $|\nabla_H u(x)|\leq 1$. Then the conclusion follows from Lemma~\ref{uniquegeod} since $u(x) - u(y) = d(x,y)$ for $\g$-a.e. $(x,y)$ (see Theorem~\ref{1lip_potential}). 
\end{proof}

The next lemma says that minimal curves used by an optimal transport plan cannot bifurcate. It follows essentially from the non-branching property of $(\hn,d)$.

\begin{lem} \label{pi1.2}
 Let $\g\in \Pi_1(\mu,\nu)$. Then $\g$ is concentrated on a set $\G$ such that the following holds. For all $(x,y)\in \G$ and $(x',y')\in \G$ such that $x\not=y$ and $x\not=x'$, if $x'$ lies on a minimal curve between $x$ and $y$ then all points $x$, $x'$, $y$ and $y'$ lie on the same minimal curve. More precisely, there exists a minimal curve $\s:[a,b]\rightarrow\hn$ such that $x=\s(a)$, $y=\s(t)$ for some $t\in \,(a,b]$, $x'=\s(s)$ for some $s\in\, (a,t]$ and $y'=\s(t')$ for some $t'\in [s,b]$.
\end{lem}

\begin{proof} Let $(x,y)\in \hn \times \hn$ and $(x',y')\in \hn \times \hn$ such that $x\not=y$ and $x'\not=x$. Assume that $x'\in \s((0,d(x,y)])$ where $\s:[0,d(x,y)] \rightarrow \hn$ is a unit-speed minimal curve between $x$ and $y$. Let $\s'$ be a unit-speed minimal curve between $x'$ and $y'$ parameterized on $[d(x,x'),d(x,x')+d(x',y')]$. Assume moreover that 
\begin{equation*} 
 d(x,y) + d(x',y') \leq d(x,y')+d(x',y).
\end{equation*}
Recall that this holds true for $\g$-a.e. $(x,y)$ and $(x',y')$ by Theorem \ref{ccycl}. Then the curve $\tilde \s :[0, d(x,x')+d(x',y')]\rightarrow\hn$ which coincides with $\s$ on $[0,d(x,x')]$ and $\s'$ on $[d(x,x'),d(x,x')+d(x',y')]$ is a length minimizing curve between $x$ and $y'$. Indeed, otherwise we would have 
\begin{equation*}
 d(x,y') < l(\tilde\s) = l(\s_{|[0,d(x,x')]}) + l(\s'_{|[d(x,x'),d(x,x')+d(x',y')]}) = d(x,x') + d(x',y').
\end{equation*}
Since $x'$ lies on a minimal curve between $x$ and $y$, we have $d(x,x') + d(x',y) = d(x,y)$ and we get 
\begin{equation*}
  d(x,y') + d(x',y) < d(x,y) + d(x',y')
\end{equation*}
which gives a contradiction. It follows that $\s$ and $\tilde\s$ are unit-speed minimal curves that coincide on the non trivial interval $[0,d(x,x')]$. Since $\hn$ is non-branching (see Proposition~\ref{nonbranching}), this implies that $\s$ and $\tilde \s$ are sub-arcs of the same minimal curve, namely $\s$ if $d(x,y') \leq d(x,y)$ and $\tilde\s$ otherwise, on which all points $x$, $x'$, $y$ and $y'$ lie. And the conclusion follows.
\end{proof}

We denote by $\Pi_2(\mu,\nu)$ the set of transport plans solution to the secondary variational problem:
\begin{equation*}
\min_{\g \in \Pi_1(\mu,\nu)} \int_{\hn\times\hn} d(x,y)^2\,d\g(x,y).
\end{equation*}

Optimal transport plans selected through the variational approximations to be introduced in Section~\ref{varapprox} will be solution to this secondary variational problem. The next lemma gives a one-dimensional monotonicity condition along minimal curves used by optimal transport plans in $\Pi_2(\mu,\nu)$. This follows essentially from a constrained version of $d^2$-cyclical monotonicity.

\begin{lem} \label{pi2}
Let $\g\in \Pi_2(\mu,\nu)$. Then $\g$ is concentrated on a set $\G$ such that the following holds. For all $(x,y)\in \G$ and $(x',y')\in \G$ such that $x\not=y$ and $x\not=x'$, if $x'$ lies on a minimal curve between $x$ and $y$ then all points $x$, $x'$, $y$ and $y'$ lie on the same minimal curve ordered in that way.
\end{lem}

In other words, there exists a minimal curve $\s:[a,b]\rightarrow\hn$ such that $\s(a) = x$, $\s(t)=y$ for some $t\in \,(a,b]$, $\s(s) = x'$ for some $s\in\, (a,t]$ and $\s(t') = y'$ for some $t'\in [t,b]$.

\begin{proof} First, as a classical fact, one can rephrase the secondary variational problem as a classical Kantorovich transport problem \eqref{e:MK} between $\mu$ and $\nu$ with cost $c(x,y) = \beta(x,y)$ with 
\begin{equation*}
\beta(x,y) = \begin{cases}
              d(x,y)^2 \quad \text{if } u(x)-u(y) = d(x,y),\\
              +\infty \qquad \phantom{\text{if}} \text{otherwise},
             \end{cases}
\end{equation*}
where $u\in \text{Lip}_1(d)$ is a Kantorovich potential associated to Kantorovich transport problem~\eqref{e:MK} between $\mu$ and $\nu$ with cost $c(x,y) = d(x,y)$ (see Section~\ref{opttrans} and Theorem~\ref{1lip_potential} there). Since $\beta$ is lower semicontinuous and $\int_{\hn\times \hn} \beta(x,y) \,d\g(x,y)<+\infty$ for all $\g \in \Pi_2(\mu,\nu)$, it follows from Theorem \ref{ccycl} that any $\g \in \Pi_2(\mu,\nu)$ is concentrated on a $\beta$-cyclically monotone set. So, taking into account the fact that $\Pi_2(\mu,\nu)\subset\Pi_1(\mu,\nu)$, we know that $\g \in \Pi_2(\mu,\nu)$ is concentrated on a set $\G$ such that 
\begin{equation*}
 u(x) - u(y) = d(x,y)
\end{equation*}
for all $(x,y)\in \G$,
\begin{equation*}
 \beta(x,y) +\beta(x',y') \leq \beta(x,y')+\beta(x',y)
\end{equation*}
for all $(x,y)\in \G$ and $(x',y')\in \G$ and the conclusion of Lemma \ref{pi1.2} holds.

Then let $(x,y)\in \G$ and $(x',y')\in \G$ be as in the statement. By Lemma \ref{pi1.2}, the conclusion will follow if we show that $d(x',y) \leq d(x',y')$. First we check that $\beta(x',y) = d(x',y)^2$ and $\beta(x,y') = d(x,y')^2$. We have 
\begin{equation*}
u(x) \leq u(x') + d(x,x') \leq u(y) + d(x',y) + d(x,x') = d(x,y) + u(y) = u(x)
\end{equation*}
hence all these inequalities are equalities. In particular, we get that $u(x') = u(y) + d(x',y)$ hence $\beta(x',y) = d(x',y)^2$. We also get that  
\begin{equation*}
 u(x) = d(x,x') + u(x') = d(x,x') + u(y') + d(x',y') = u(y') + d(x,y')
\end{equation*}
hence $\beta(x,y') = d(x,y')^2$. If $d(x',y')< d(x',y)$, we  get
\begin{equation*}
\begin{split}
 \beta(x,y') + &\beta(x',y) - \beta(x',y') - \beta(x,y)\\
&= d(x,y')^2  + d(x',y)^2 - d(x',y')^2 - d(x,y)^2 \\
&= (d(x,x')+d(x',y'))^2 + d(x',y)^2 - d(x',y')^2 - (d(x,x')+d(x',y))^2\\
&= 2 \, d(x,x') (d(x',y') - d(x',y))<0
\end{split}
\end{equation*}
which gives a contradiction.
\end{proof}

%%%%
%%%%%%%%%%%%%%%%%%%%%%%
\section{Variational approximations} \label{varapprox}

We introduce variational approximations in the spirit of \cite{ap} (see also \cite{cfm}, \cite{akp}) by rephrasing in our geometrical context the variational approximations considered recently in \cite{sant}. This approximation procedure will be used to select optimal transport plans that will be eventually proved to be induced by transport maps.

Let $\mu$, $\nu \in \p_c(\hn)$ be fixed. Let $K$ be a compact subset of $\hn$ such that $\spt{\mu}\cup\spt{\nu}\subset K$ and set
\begin{equation*}
 \Pi := \{\g \in \p(\hn\times\hn);\, (\pi_1)_\sharp\g=\mu,\, \spt{(\pi_2)_\sharp\g} \subset K\}.
\end{equation*}
For $\e>0$ fixed and $\g \in \Pi$, we set
\begin{multline*}
 C_\e(\g) := 
\frac{1}{\e}\, W_1((\pi_2)_\sharp\g,\nu) + \int_{\hn\times\hn} d(x,y)\,d\g(x,y) \\ +  \e \int_{\hn\times\hn} d(x,y)^2\,d\g(x,y) + \e^{6n+8} \card{(\spt{(\pi_2)_\sharp\g})}
\end{multline*}
and consider the family of minimization problems:
\begin{equation} \label{e:varapprox}
 \min\{ C_\e(\g);\, \g \in \Pi\}. \tag{$P_\e$}
\end{equation}
Here $W_1$ denotes the 1-Wasserstein distance defined for any two probability measures $\mu_1$, $\mu_2 \in \p(\hn)$ by 
\begin{equation*}
 W_1(\mu_1,\mu_2) := \min_{\g \in \Pi(\mu_1,\mu_2)} \int_{\hn\times\hn} d(x,y)\,d\g(x,y).
\end{equation*}
First we note that \eqref{e:varapprox} always admits solutions.

\begin{thm}
 For any $\e>0$, the problem \eqref{e:varapprox} admits at least one solution and $\min\{ C_\e(\g);\, \g \in \Pi\}<+\infty$.
\end{thm}

\begin{proof}
First note that since $K$ is compact, $C_\e(\g)<+\infty$ for any $\g\in \Pi$ such that $(\pi_2)_\sharp\g$ is finitely atomic. Next the existence of solutions to \eqref{e:varapprox} follows from the weak compactness of $\Pi$, the lower semicontinuity of the three first terms to be minimized and the Kuratowski convergence of the supports of weakly converging probability measures (see \cite[Chapter~5]{ags}).
\end{proof}

Next, weak limits of solutions to \eqref{e:varapprox} are optimal transport plans that are solutions to the secondary variational problem introduced in Section~\ref{sect:optplanning} to which we refer for the definition of $\Pi_2(\mu,\nu)$. Modulo minor modifications due to our geometrical context, this can be proved with the same arguments as those given in \cite{sant}. 

\begin{lem} \label{optpi2}
 Let $\e_k$ be a sequence converging to 0 and $\g_{\e_k}$ a sequence of solutions to $(P_{\e_k})$ which is weakly converging to some $\g\in \p(\hn\times\hn)$. Then $\g \in \Pi_2(\mu,\nu)$.
\end{lem}

\begin{proof} First we note that for any $m\geq 1$, one can find a finite set $F_m\subset K$ such that $\card F_m \leq C\, m^{2n+2}$ for some constant $C>0$ which depends only on $n$ and $\diam K$ and a Borel map $p_m:K \rightarrow F_m$ such that
\begin{equation*}
 d(p_m(x),x) < 1/m  
\end{equation*}
for all $x\in K$. Indeed choose $x_1 \in K$. For $i\geq 2$, choose by induction $x_i \in K\setminus \cup_{j<i} B(x_j,1/m)$ as long as $K\setminus \cup_{j<i} B(x_j,1/m) \not=\emptyset$. Let $F_m$ denote the set of all these points. The balls $B(x_i,1/(2m))$ are mutually disjoint and $\cup_{i} B(x_i,(1/(2m)) \subset B(x_1,\diam K + 1)$. Remembering \eqref{e:measball}, it follows that
\begin{equation*}
\begin{split}
 c_n (2m)^{-2n-2} \card F&\leq \Leb(\cup_i B(x_i,1/(2m)))\\
&\leq \Leb(B(x_1,\diam K +1)) = c_n (\diam K + 1)^{2n+2}
\end{split}
\end{equation*}
for any finite subset $F\subset F_m$, hence $F_m$ is a finite set with $\card F_m \leq C\, m^{2n+2}$ where $C$ depends only on $n$ and $\diam K$. Next, by construction, for any $x \in K$, there exists a unique $x_i\in F_m$ such that $x \in B(x_i,1/m)\setminus \cup_{j<i} B(x_j,1/m)$ and we then set $p_m(x) := x_i$. 

The proof of the lemma can now be completed following the same arguments as those in \cite{sant}. For sake of completeness, we sketch these arguments below. Let $\g_{\e_k}$ and $\g$ be as in the statement. For $m\geq 1$, one sets $\nu_m:= (p_m)_\sharp\nu$. Note that by construction of $F_m$ and $p_m$, one has $\card{(\spt{\nu_m})} \leq  C\, m^{2n+2}$ and $W_1(\nu_m,\nu) \leq 1/m$. To check that $\g \in \Pi(\mu,\nu)$, one takes some $\g_m \in \Pi(\mu,\nu_m)\subset \Pi$ and uses the optimality of $\g_{\e_k}$ which implies 
\begin{equation*}
\begin{split}
 W_1((\pi_2)_\sharp\g_{\e_k},\nu) &\leq \e_k \, C_{\e_k}(\g_m) \\
&\leq \dfrac{1}{m} + \e_k \diam(K) + \e_k^2 \diam(K)^2 + C \e_k^{6n+9} m^{2n+2}.
\end{split}
\end{equation*}
Then one lets $\e_k \rightarrow 0$ with $m\geq 1$ fixed and then $m \rightarrow +\infty$ to get that $(\pi_2)_\sharp\g_{\e_k}$ converges weakly to $\nu$. Since it also converges weakly to $(\pi_2)_\sharp\g$, it follows that $\g \in \Pi(\mu,\nu)$. 

To check that $\g \in \Pi_1(\mu,\nu)$, one takes $\overline\g \in \Pi_1(\mu,\nu)$ and sets $\overline\g_m := (Id,p_m)_\sharp\overline\g \in \Pi(\mu,\nu_m)$ where $(Id,p_m)(x,y) = (x,p_m(y))$. By optimality of $\g_{\e_k}$, one has 
\begin{equation*}
\begin{split}
 \int_{\hn\times\hn} d(x,y)&\,d\g_{\e_k}(x,y) \leq  C_{\e_k}(\overline\g_m)\\
&\leq \dfrac{1}{m \,\e_k} + \int_{\hn\times\hn} d(x,y)\,d\overline\g_m(x,y) + \e_k \diam(K)^2 + C \e_k^{6n+8} m^{2n+2}.
\end{split}
\end{equation*}
Choosing $m$ of the order of $\e_k^{-2}$ and letting $\e_k \rightarrow 0$, one gets 
\begin{equation*}
 \int_{\hn\times\hn} d(x,y)\,d\g(x,y) \leq  \int_{\hn\times\hn} d(x,y)\,d\overline\g(x,y)
\end{equation*}
hence $\g \in \Pi_1(\mu,\nu)$. 

Finally, to check that $\g \in \Pi_2(\mu,\nu)$, one uses once again the optimality of $\g_{\e_k}$ in the following way,
\begin{multline*}
  W_1((\pi_2)_\sharp\g_{\e_k},\nu) + \e_k \int_{\hn\times\hn} d(x,y)\,d\g_{\e_k}(x,y) + \e_k^2 \int_{\hn\times\hn} d(x,y)^2\,d\g_{\e_k}(x,y) \\\leq \e_k \, C_{\e_k}(\overline\g_m) 
\end{multline*}
On the other hand, one has 
\begin{gather*}
\begin{split}
 W_1(\mu,\nu) &\leq W_1(\mu,(\pi_2)_\sharp\g_{\e_k}) + W_1((\pi_2)_\sharp\g_{\e_k},\nu)\\
& \leq \int_{\hn\times\hn} d(x,y)\,d\g_{\e_k}(x,y)  + W_1((\pi_2)_\sharp\g_{\e_k},\nu)
\end{split}\\
\intertext{and}
\begin{split}
\int_{\hn\times\hn} d(x,y)\,d\overline\g_m(x,y) &= \int_{\hn\times\hn} d(x,p_m(y))\,d\overline\g(x,y)\\
&\leq \int_{\hn\times\hn} d(x,y)\,d\overline\g(x,y) + \int_{\hn\times\hn} d(y,p_m(y))\,d\overline\g(x,y) \\
&\leq  W_1(\mu,\nu) + \frac{1}{m}.
\end{split}
\end{gather*}
It follows that 
\begin{multline*}
 \int_{\hn\times\hn} d(x,y)^2\,d\g_{\e_k}(x,y) \\ \leq \frac{1}{m\, \e_k^2} + \frac{1}{m\, \e_k} + \int_{\hn\times\hn} d(x,y)^2\,d\overline\g_m(x,y) + C \e_k^{6n+7} m^{2n+2}
\end{multline*}
provided $\e_k\leq 1$. Choosing $m$ of the order of $\e_k^{-3}$ and letting $\e_k \rightarrow 0$, one gets
\begin{equation*}
 \int_{\hn\times\hn} d(x,y)^2\,d\g(x,y) \leq \int_{\hn\times\hn} d(x,y)^2\,d\overline\g(x,y).
\end{equation*}
Since $\overline\g \in \Pi_1(\mu,\nu)$ was arbitrary, it follows that $\g \in \Pi_2(\mu,\nu)$.
\end{proof}
\medskip

The rest of this section is devoted to the study of the solutions to ($P_\e$). We fix $\e>0$ and set 
\begin{equation*}
 c_\e (x,y) := d(x,y) + \e\, d(x,y)^2.
\end{equation*}
We first recall the following classical fact.

\begin{lem} \label{restrictions}
Let $\g_\e$ be a solution to ($P_\e$). Then for any Borel set $U\subset\hn\times\hn$, $(\pi_2)_\sharp(\g_\e\lfloor U)$ is finitely atomic and  $\g_\e\lfloor U$ is a solution to Kantorovich transport problem \eqref{e:MK} between $(\pi_1)_\sharp (\g_\e\lfloor U)$ and $(\pi_2)_\sharp (\g_\e\lfloor U)$ with  cost $c_\e$.
\end{lem}

\begin{proof}
 The fact that $(\pi_2)_\sharp(\g_\e\lfloor U)$ is finitely atomic obviously follows from the fact that $C_\e(\g_\e) = \min\{ C_\e(\g);\, \g \in \Pi\}<+\infty$. Next it is also immediate that $\g_\e$ is a solution to Kantorovich transport problem \eqref{e:MK} between $\mu$ and $(\pi_2)_\sharp (\g_\e)$ with cost $c_\e$. Then as a classical fact, the claim follows from the linearity of the functional to be minimized with respect to the transport plan. If $\g \in \Pi((\pi_1)_\sharp (\g_\e\lfloor U),(\pi_2)_\sharp(\g_\e\lfloor U))$, one indeed simply compare $C_\e(\g_\e)$ with $C_\e(\hat\g)$ where $ \hat\g = \g_\e\lfloor (\hn\times\hn)\setminus U + \g \in \Pi(\mu,(\pi_2)_\sharp(\g_\e))$ to get the conclusion.
\end{proof}

Next in this section we consider interpolations between two measures $\overline\mu$, $\overline \nu\in \p_c(\hn)$ that are constructed from a transport plan solution to Kantorovich transport problem \eqref{e:MK} between these two measures with cost $c_\e$. We prove absolute continuity and, more importantly, $L^\infty$-estimates on the density with respect to $\Leb$ of these interpolations whenever $\overline\mu \ll \Leb$ and $\overline \nu$ is finitely atomic, see Proposition~\ref{e:Linftydensityestimates}. We divide the arguments into several steps. First we prove that any solution to this Kantorovich transport problem is induced by a transport.

\begin{thm}  \label{existenceceps}
 Let $\overline\mu$, $\overline \nu\in \p_c(\hn)$ be fixed. Assume that $\overline\mu \ll \Leb$ and that $\overline \nu$ is finitely atomic. Then any solution to Kantorovich transport problem \eqref{e:MK} between $\overline\mu$ and $\overline \nu$ with  cost $c_\e$ is induced by a transport. In particular there exists a unique optimal transport map solution to the transport problem \eqref{e:M} between $\overline\mu$ and $\overline \nu$ with cost $c_\e$.
\end{thm}

\begin{proof} Let $\psi$ be a Kantorovich potential for Kantorovich transport problem \eqref{e:MK} between $\overline\mu$ and $\overline \nu$ with  cost $c_\e$ given by Theorem~\ref{duality}. Let $\{y_i\}_{i=1}^k$ denote the atoms of $\overline \nu$. We prove that for $\Leb$-a.e. $x\in\hn$, there is at most one point $y_i$ for some $i\in\{1,\dots,k\}$ such that
\begin{equation*} 
 \psi(x) + \psi^c(y_i) = c_\e (x,y_i).
\end{equation*}
Since $\overline\mu \ll \Leb$, it will follow that any transport plan solution to Kantorovich transport problem \eqref{e:MK} between $\overline\mu$ and $\overline \nu$ with  cost $c_\e$ is concentrated on a $\overline\mu$-measurable graph and hence induced by a transport. This implies in turn existence and uniqueness of the optimal transport map solution to the transport problem \eqref{e:M} between $\overline\mu$ and $\overline \nu$ with cost $c_\e$ (see Theorem~\ref{plan-transport}). 

For $i\not= j$, set $h_{ij}(x) := c_\e (x,y_i) - c_\e (x,y_j) + \psi^c(y_j) - \psi^c(y_i)$. It follows from Lemma~\ref{prop-distcc} that $h_{ij}$ is of class $C^\infty$ on the open set $\hn\setminus(L_{y_i} \cup L_{y_j})$ with $\nabla h_{ij} \not= 0$. Indeed assume on the contrary that $\nabla h_{ij}(x) =0$ for some $x\in \hn\setminus(L_{y_i} \cup L_{y_j})$. Then, differentiating along the horizontal vector fields $X_j$ and $Y_j$, we would have
\begin{equation*} 
 \nabla_H d_{y_i} (x) \,(1+2\e\, d_{y_i}(x)) =  \nabla_H d_{y_j} (x) \,(1+2\e \,d_{y_j}(x)).
\end{equation*}
Since $|\nabla_H d_{y_i} (x) | = |\nabla_H d_{y_j} (x) |$ (see Lemma~\ref{prop-distcc}(i)), this would imply that $d_{y_i}(x) = d_{y_j}(x)$ and in turn that $\nabla_H d_{y_i} (x) = \nabla_H d_{y_j} (x)$. Since we also have by assumption $\partial_t d_{y_i} (x) = \partial_t d_{y_j} (x)$, Lemma~\ref{prop-distcc}(ii) would give $y_i=y_j$. It follows that the set $\{x\in\hn\setminus(L_{y_i} \cup L_{y_j});\, h_{ij}(x) =0\}$ is a $C^\infty$-smooth submanifold of dimension $2n$ in $\R^{2n+1}$ and hence has Lebesgue measure 0. Since $\Leb(L_{y_i}) =0$, it follows that 
\begin{equation*} 
 \Leb (\bigcup_{i\not=j} \{x\in\hn;\,h_{ij}(x) =0\})  = 0
\end{equation*}
and the claim follows.
\end{proof}

If $T:\hn \rightarrow \hn$, we set $T_t = e_t \circ S \circ (I \otimes T)$, i.e., $T_t(x)$ is the point lying at distance $t \, d(x,T(x))$ from $x$ on the selected minimal curve $S(x,T(x))$ between $x$ and $T(x)$ (see Subsection~\ref{sect-interpolation} for the definition of $S$ and $e_t$).

\begin{prop} \label{injectivity} \cite[Chapter 7]{villani}
 Let $\overline\mu$, $\overline \nu\in \p_c(\hn)$ be fixed such that $\overline\mu \ll \Leb$ and $\overline \nu$ is finitely atomic. Let $T^\e$ be the optimal transport map solution to the transport problem \eqref{e:M} between $\overline\mu$ and $\overline \nu$ with cost $c_\e$. Then there exists a $\overline\mu$ - measurable set $A$ such that $\overline\mu(A) = 1$ and such that for each $t\in\, [0,1)$, $T^\e_t\lfloor_A$ is injective.
\end{prop}

The cost $c_\e$ can be recovered as coming from a so-called coercive Lagrangian action. Since $(\hn,d)$ is non-branching, the proposition essentially follows from~\cite[Chapter 7, Theorem 7.30]{villani}. However one does not need the full strength of the theory developed in~\cite[Chapter 7]{villani} to get the conclusion of Proposition~\eqref{injectivity} and we sketch below the arguments for the reader's convenience. 

\begin{proof} Let $0\leq s<t \leq 1$ and $x$, $y\in \hn$. Set 
\begin{equation*} 
c_\e^{s,t} (x,y) = d(x,y) + \e \, \dfrac{d(x,y)^2}{t-s}.
\end{equation*}
The space $(\hn,d)$ being a geodesic space and $u \mapsto u + \e\,u^2$ being strictly increasing and strictly convex on $[0,+\infty )$, one has 
\begin{equation*} 
 c_\e (x,y) \leq c_\e^{0,t} (x,z) + c_\e^{t,1} (z,y)
\end{equation*}
for all $x$, $y$, $z\in \hn$ and $t\in (0,1)$, with equality if and only if any curve in $C([0,1],\hn)$ obtained by concatenation of a minimal curve $\s_{x,z}:[0,t]\rightarrow\hn$ between $x$ and $z$ and a minimal curve $\s_{z,y}:[t,1]\rightarrow\hn$ between $z$ and $y$ is a minimal curve between $x$ and $y$.

On the other hand, by $c_\e$-cyclical monotonicity, one knows that there exists a $\overline\mu$ - measurable set $A$ such that $\overline\mu(A) =1$ and 
\begin{equation*} 
 c_\e(x,T^\e(x)) + c_\e(\tilde x,T^\e(\tilde x)) \leq c_\e(x,T^\e(\tilde x)) + c_\e(\tilde x,T^\e(x))
\end{equation*}
for all $x$, $\tilde x \in A$ (see Theorem~\ref{ccycl}).

Now let $t\in (0,1)$ be fixed and let $x$, $\tilde x \in A$. Assume that $T^\e_t(x) = T^\e_t(\tilde x)$. Then 
\begin{equation} 
 c_\e(x,T^\e(\tilde x)) \leq c_\e^{0,t} (x,z) + c_\e^{t,1} (z,T^\e(\tilde x)),\label{e:injectivity}
\end{equation}
and similarly,
\begin{equation*} 
c_\e(\tilde x,T^\e(x)) \leq c_\e^{0,t} (\tilde x,z) + c_\e^{t,1} (z,T^\e(x))
\end{equation*}
where $z = T^\e_t(x) = T^\e_t(\tilde x)$. It follows that 
\begin{equation*} 
\begin{split}
c_\e(x,& T^\e(x)) + c_\e(\tilde x,T^\e(\tilde x))\\ &\leq c_\e(x,T^\e(\tilde x)) + c_\e(\tilde x,T^\e(x))\\
&\leq c_\e^{0,t} (x,T^\e_t(x)) + c_\e^{t,1} (T^\e_t(x),T^\e(x)) + c_\e^{0,t} (\tilde x,T^\e_t(\tilde x)) + c_\e^{t,1} (T^\e_t(\tilde x),T^\e(\tilde x))\\
&= c_\e(x,T^\e(x)) + c_\e(\tilde x,T^\e(\tilde x)).
\end{split}
\end{equation*}
Hence equality has to hold in all these inequalities. In particular equality holds in \eqref{e:injectivity}. It follows that the curve obtained by concatenation of the minimal curve $s\in [0,t] \mapsto e_s(S(x,T^\e(x)))$ between $x$ and $z$ with the minimal curve $s \in [t,1] \mapsto e_s(S(\tilde x,T^\e(\tilde x)))$ between $z$ and $T^\e(\tilde x)$ is a minimal curve. Since this curve coincide with the minimal curve $\s: s\in [0,1] \mapsto e_s(S(x,T^\e(x)))$ on the non trivial interval $[0,t]$ and $\hn$ is non-branching (see Proposition~\ref{nonbranching}), we get that it coincides with $\s$ on the whole interval $[0,1]$. Similarly, it coincides with the minimal curve $\tilde \s: s\in [0,1] \mapsto e_s(S(\tilde x,T^\e(\tilde x)))$ on the whole interval $[0,1]$. Hence $\s = \tilde \s$ and in particular $x=\s(0) = \tilde \s (0) = \tilde x$.
\end{proof}

We turn now to the main estimate that will lead to Proposition~\ref{e:Linftydensityestimates}.

\begin{prop} \label{estimsup}
 Let $\overline\mu \in \p(\hn)$, $\overline\mu \ll \Leb$, $\overline\mu = \rho \,d\Leb$, and $T:\hn\rightarrow\hn$ be a $\overline\mu$ - measurable map such that $T_\sharp\overline\mu$ is finitely atomic. Let $t\in\,(0,1)$ and set $\overline\mu_t := T_{t\,\sharp}\overline\mu$. Assume that there exists a $\overline\mu$ - measurable set $A$ such that $\overline\mu(A) = 1$ and $T_t\lfloor_A$ is injective. Then $\overline\mu_t \ll \Leb$, $\overline\mu_t = \rho_t \,d\Leb$ with 
\begin{equation*}
 \rho_t \leq \frac{\1_{T_t(A)}}{(1-t)^{2n+3}} \,\rho \circ T_t^{-1}\lfloor_{T_t(A)} \qquad \Leb-a.e.
\end{equation*}
\end{prop}

Arguments for the proof of this proposition can be found in \cite[Section~3]{figjuil} even-though not explicitly stated in the same way in that paper. They rely on the following estimate:
\begin{equation*}
 \Leb(E) \leq \frac{1}{(1-t)^{2n+3}} \,\Leb((e_t \circ S)(E,y))
\end{equation*}
for any $y\in\hn$ and $E\subset \hn$ which is proved in \cite[Section~2]{juillet} and which roughly means that $(\hn,d,\Leb)$ satisfies a  so-called Measure Contraction Property. We detail the proof below for the reader's convenience.

\begin{proof} 
 Let $\{y_i\}_{i=1}^k$ denote the atoms of $T_\sharp\overline\mu$. Set $A_i = T^{-1}(\{y_i\}) \cap A$ and $\hat A = \cup_i A_i$. The sets $A_i$ are mutually disjointed and $\overline\mu(\hat A)=1$ by hypothesis. For any $x\in A_i$, $T_t(x) = (e_t \circ S) (x,y_i)$, hence
\begin{equation*}
 \Leb(E) \leq \frac{1}{(1-t)^{2n+3}} \,\Leb(T_t(E))
\end{equation*}
for any $E\subset A_i$. Next if $E\subset \hat A$, writing $E = \cup_i (E\cap A_i)$ where the sets $A_i$ are mutually disjointed and remembering that $T_t$ is injective on $\hat A\subset A$ by hypothesis, one gets
\begin{equation*}
  \Leb(E) \leq \frac{1}{(1-t)^{2n+3}}\, \Leb(T_t(E)).
\end{equation*}
It follows that for any $F\subset \hn$,
\begin{equation*} \label{e:1}
 \Leb(T_t^{-1}(F)\cap \hat A) \leq \frac{1}{(1-t)^{2n+3}} \Leb(F\cap T_t(\hat A)).
\end{equation*}

Assume that $F\subset \hn$ is such that $\Leb(F)=0$. We get $\Leb(T_t^{-1}(F)\cap \hat A) = 0$ from the previous inequality. On the other hand $\overline\mu_t(F) = \overline \mu (T_t^{-1}(F)) = \overline \mu (T_t^{-1}(F)\cap \hat A)$ hence $\overline\mu_t(F)=0$ since $\overline\mu\ll\Leb$ and it follows that $\overline\mu_t \ll \Leb$. 

Next, to prove the estimate on the density of $\overline\mu_t$ with respect to $\Leb$, we note that the inequality above implies that 
\begin{equation*}
 \int_{\hat A} g(T_t) \, d\Leb \leq \frac{1}{(1-t)^{2n+3}} \int_{T_t(\hat A)} g \, d\Leb
\end{equation*}
for any non negative measurable map $g:\hn\rightarrow [0,+\infty]$. Let $h:\hn\rightarrow [0,+\infty]$ be a non negative measurable map and set 
\begin{equation*}
 g(x) = \1_{T_t(\hat A)}(x) \, h(x) \, \rho(T_t^{-1}\lfloor_{T_t(\hat A)}(x)).
\end{equation*}
Then 
\begin{equation*}
 \int_{\hat A} g(T_t(x)) \, d\Leb(x) \leq \frac{1}{(1-t)^{2n+3}} \int_{T_t(\hat A)} h(x) \,\rho(T_t^{-1}\lfloor_{T_t(\hat A)}(x)) \, d\Leb(x)
\end{equation*}
On the other hand
\begin{equation*}
 \int_{\hat A}  g(T_t) \, d\Leb = \int_{\hat A} h(T_t)\, \rho \, d\Leb = \int_{\hn} h(T_t)\, d\overline\mu = \int_{\hn} h\, d\overline\mu_t,
\end{equation*}
hence
\begin{equation*}
 \int_{\hn} h(x)\, d\overline\mu_t(x) \leq \frac{1}{(1-t)^{2n+3}} \int_{T_t(\hat A)} h(x) \, \rho(T_t^{-1}\lfloor_{T_t(\hat A)}(x)) \, d\Leb(x).
\end{equation*}
Remembering that $\hat A \subset A$, this concludes the proof.
\end{proof}

Finally, combining Theorem~\ref{existenceceps}, Propositions~\ref{injectivity} and \ref{estimsup}, we get the following proposition which gives the absolute continuity of the interpolations together with an $L^\infty$-estimate on their density. Note that if $\g_\e$ is the transport plan solution to Kantorovich transport problem \eqref{e:MK} between $\overline\mu$ and $\overline \nu$ with cost $c_\e$ and $T^\e$ the optimal transport map solution to the transport problem \eqref{e:M} between $\overline\mu$ and $\overline \nu$ with cost $c_\e$, which hence induces $\g_\e$, then $(e_t \circ S)_\sharp\g_\e=T_{t\,\sharp}^\e\overline\mu$.

\begin{prop} \label{e:Linftydensityestimates}
 Let $\overline\mu$, $\overline \nu\in \p_c(\hn)$ be fixed. Assume that $\overline\mu \ll \Leb$ with $\overline\mu = \rho \,d\Leb$ and $\overline \nu$ is finitely atomic. Let $\g_\e$ be the transport plan solution to Kantorovich transport problem \eqref{e:MK} between $\overline\mu$ and $\overline \nu$ with cost $c_\e$. Then for any $t\in [0,1)$, the interpolation $(e_t \circ S)_\sharp\g_\e$ is absolutely continuous with respect to $\Leb$, $(e_t \circ S)_\sharp\g_\e = \rho_t^\e \,d\Leb$, and one has 
\begin{equation*}
 \|\rho_t^\e\|_{L^\infty} \leq \frac{1}{(1-t)^{2n+3}}\, \|\rho\|_{L^\infty}.
\end{equation*}
\end{prop}

%%%%
%%%%%%%%%%%%%%%%%%%%%%%
\section{Properties of measures $\g\in \p(\hn\times\hn)$ with $(\pi_1)_\sharp\g \ll \Leb$} \label{lebpoints}

This section is independent of the transport problem. We state some properties of measures in $\p(\hn\times\hn)$ with first marginal absolutely continuous with respect to $\Leb$. These properties are essential steps in the strategy adopted here to solve Monge's transport problem. They are the exact counterpart in the framework of $(\hn, d, \Leb)$ of similar properties proved in~\cite{champion-dePascale} in $\Rn$. These properties hold actually true in more general settings, for instance in any separable doubling metric measure space.

We first recall some facts about Lebesgue points of Borel functions and density of absolutely continuous measures. Since the measure $\Leb$ is a doubling measure on $(\hn,d)$, see \eqref{e:measball}, if $\rho:\hn \rightarrow [0,+\infty]$ is a $\Leb$-locally summable Borel function then for $\Leb$-a.e. $x\in \hn$, one has 
\begin{equation} \label{e:lebpoint}
 \lim_{r\rightarrow 0} \dfrac{1}{\Leb(B(x,r))} \, \int_{B(x,r)} |\rho(y) - \rho(x)|\,d\Leb(y) = 0,
\end{equation}
see e.g. \cite{heinonen}. A point $x\in \hn$ where \eqref{e:lebpoint} holds is called a Lebesgue point of $\rho$ and we denote by $\leb \rho$ the set of all Lebesgue points of $\rho$. 

In rest of this paper, and especially in the lemma to follow, it will be technically convenient to consider the density $\rho$ of an absolutely continuous measure $\mu \ll \Leb$ as a $\Leb$-summable Borel function, i.e., a function well-defined everywhere, so that one can speak about its Lebesgue points and its value at any arbitrary point without any ambiguity. If $\mu\in \p(\hn)$, we set 
\begin{equation} \label{e:density}
 \rho(x) := \limsup_{r\rightarrow 0} \dfrac{\mu(B(x,r))}{\Leb(B(x,r))}\,.
\end{equation}
This map $\rho:\hn \rightarrow [0,+\infty]$ is a Borel map and if $\mu\ll \Leb$ then $\mu = \rho \,d\Leb$. By a slight abuse of terminology, when speaking about the density of an absolutely continuous measure $\mu\in \p(\hn)$ with respect to $\Leb$, we will thus always refer in the following to the Borel function $\rho$ defined above.

The next lemma will be an essential ingredient in the proof of Lemma~\ref{mainlemma}. 

\begin{lem} \label{dens1}
 Let $\g\in \p(\hn\times\hn)$ be such that $(\pi_1)_\sharp\g \ll \Leb$. Then $\g$ is concentrated on a set $\G$ such that, for all $(x,y)\in  \G$ and all $r>0$, there exist $y'\in\hn$ and $r'>0$ such that

\renewcommand{\theenumi}{\roman{enumi}}
\begin{enumerate}
 \item $y \in B(y',r') \subset\subset B(y,r)$,

 \item $x\in \leb{\rho}$ and $\rho(x) <+\infty$,

 \item $x\in \leb{\rho'}$ and $\rho'(x)>0$,
\end{enumerate}
where $\rho$ denotes the density of $(\pi_1)_\sharp\g$ and $\rho'$ the density of $(\pi_1)_\sharp\g\lfloor(\hn\times B(y',r'))$ with respect to $\Leb$.
\end{lem}

\begin{proof} Let $(y_m)_{m\geq 1}$ be a dense sequence in $\hn$.
For each $m,k \in \N^*$, set $\g_{m,k}:=\gamma \lfloor (\hn \times B(y_m, r_k))$ where $r_k:=1/k$. Let $\rho_{m,k}$ denote the density of $(\pi_1)_\sharp\g_{m,k}$ with respect to $\Leb$. Set $A_{m,k} := \hn \setminus (\leb \rho \cap \leb \rho_{m,k} \cap \{\rho < +\infty\})$. We have $\Leb(A_{m,k})=0$. Since $(\pi_1)_\sharp\g \ll \Leb$, it follows that $\g(A_{m,k}\times B(y_m, r_k)) \leq (\pi_1)_\sharp\g (A_{m,k}) =0$. Next
 \begin{equation*}
  \g(\{\rho_{m,k}=0\} \times B(y_m,r_k)) = (\pi_1)_\sharp\g_{m,k}(\{\rho_{m,k}=0\}) = 0.
 \end{equation*}
It follows that $\g(D_{m,k})=0$ for all $m,k \in \N^*$ where
\begin{equation*}
 D_{m,k}:=\left[\hn \setminus (\leb \rho \cap \leb \rho_{m,k} \cap \{\rho < +\infty\} \cap \{\rho_{m,k}>0\}) \right] \times B(y_m,r_k)
\end{equation*}
hence $\g(\cup_{m,k} D_{m,k})=0$ and $\g$ is concentrated on $\hn \setminus \cup_{m,k} D_{m,k}$. Then the conclusion follows noting that for each $(x,y) \in \hn\times\hn$ and $r>0$, one can find $m,k\in \N^*$ such that $y \in B(y_m,r_k)\subset \subset B(y,r)$.
\end{proof}

We say that $x\in E$ is a Lebesgue point of a Borel set $E$ if $x\in \leb \1_E$, i.e., if $x\in E$ and 
\begin{equation*}
 \lim_{r\rightarrow 0} \dfrac{\Leb(E\cap B(x,r))}{\Leb(B(x,r))} = 1,
\end{equation*}
and we denote by $\leb E:= \leb \1_E$ the set of all Lebesgue points of $E$. Note that $\Leb(E\setminus \leb E) = 0$.

The next lemma together with Lemma~\ref{mainlemma} and Lemma~\ref{pi2} is one of the key ingredients of the proof of Theorem~\ref{mainthmbis} and eventually of the existence of a solution to Monge's transport problem. It can be recovered as a consequence of Lemma~\ref{dens1}. However, for sake of clarity, we state and prove it independently.

\begin{lem}  \label{dens2}
 Let $\g\in \p(\hn\times\hn)$ be such that $(\pi_1)_\sharp\g \ll \Leb$. Assume that $\g$ is concentrated on a $\s$-compact set $\G$. For $y\in\hn$ and $r>0$, set 
\begin{equation*}
 \G^{-1} (B(y,r)) = \pi_1(\G \cap (\hn \times B(y,r))).
\end{equation*}
Then $\G^{-1} (B(y,r))$ is a Borel set and $\g$ is concentrated on a set $\G'\subset\G$ such that for all $(x,y)\in\G'$ and all $r>0$, $x\in \leb{\G^{-1} (B(y,r)})$.
\end{lem}

\begin{proof}
 Since $\G$ is  $\s$-compact, $\G^{-1} (B(y,r))$ is also $\s$-compact hence a Borel set. Set $A := \{(x,y) \in \G;\,x\notin \leb\Gamma^{-1}(B(y,r)) \text{ for some } r>0 \}$
and let us show that $\g(A)=0$. For each $k\in\N^*$, consider a countable covering of 
$\hn$ by balls $(B(y_i^k,r_k))_{\,i \geq 1}$ of radius $r_k:=1/(2k)$. If $(x,y) \in \G$ and $x \notin \leb\Gamma^{-1}(B(y,r))$ then for any $k \geq 1/r$ and $y_i^k$ such that
$d(y_i^k,y)<r_k$, one has $x\in \G^{-1}(B(y_i^k,r_k)) \setminus \leb\G^{-1}(B(y_i^k,r_k))$. It follows that 
\begin{equation*}
\pi_1(A) \,\, \subset \,\, \bigcup_{k\geq 1} \, \bigcup_{\,i \geq 1} 
\G^{-1}(B(y_i^k,r_k)) \setminus \leb\G^{-1}(B(y_i^k,r_k)) .
\end{equation*}
The set on the right-hand side has $\Leb$-measure 0. Since $(\pi_1)_\sharp\g \ll \Leb$, it follows that $\g(A) \leq  (\pi_1)_\sharp\g(\pi_1(A))=0$.
\end{proof}

%%%%%%%%%%%%%%%%%%%%%%%
\section{Lower density of the transport set} \label{main}

We consider optimal transport plans in $\Pi_1(\mu,\nu)$ that are obtained as weak limit of solutions to the variational approximations introduced in Section~\ref{varapprox}. We prove that if $\g$ is such a transport plan then it is concentrated on a set $\G$ whose related transport set has positive lower density at each point $x\in \pi_1(\G)$ for some suitable notion of lower density. As already mentioned, this is one of the main ingredient in the proof of Theorem~\ref{mainthmbis}. Following the notion of transport set introduced in e.g.~\cite{ap}, we define in our geometrical context the transport set related to a set $\G\subset \hn\times\hn$ as 
\begin{equation*}
 T(\G) := \{ (e_t \circ S)(x,y);\, (x,y)\in \G, \, t\in (0,1) \}.
\end{equation*}
Recall that $S$ is a measurable selection of minimal curves and that $(e_t \circ S)(x,y)$ denotes the point at distance $t \,d(x,y)$ from $x$ on the selected minimal curve $S(x,y)$ between $x$ and $y$, see Subsection~\ref{sect-interpolation}.

\begin{lem}  \label{mainlemma}
 Let $\g\in\Pi(\mu,\nu)$ obtained as a weak limit of solutions to $(P_{\e_k})$ for some sequence $\e_k$ converging to 0. Then $\g$ is concentrated on a set $\G$ such that for all $(x,y)\in\G$ such that $x\not=y$ and all $r>0$, we have
\begin{equation*}
 \liminf_{\delta\downarrow 0} \, \dfrac{\Leb(T(\G \cap [B(x,\frac{\delta}{2}) \times B(y,r)]) \cap B(x,\delta))}{\Leb(B(x,\delta))}\, > 0.
\end{equation*}
\end{lem}

The proof below follows the line of the proof of the similar property in~\cite{champion-dePascale}. In our context it requires however some technical refinement. 

 \begin{proof}  We consider the set $\Gamma$ obtained by Lemma~\ref{dens1}, $(x,y) \in \Gamma$ with $x \neq y$ and $r>0$. Then let $y'\in\hn$ and $r'>0$ be given by Lemma~\ref{dens1} so that Lemma~\ref{dens1}(i), (ii) and (iii) hold. Using the same notations as in this lemma, we set \begin{equation*}
 G:=\{z \in \hn;\, \ \frac{1}{2}\rho'(x) \leq \rho' (z) \,\, \mbox{and} \,\,
\rho(z) \leq  2\rho(x) \}.
\end{equation*}
Then $G$ is a Borel set. We have $0< \rho'(x) \leq \rho(x)$ (remember the convention about densities of absolutely continuous measure, see~\eqref{e:density}). Since $x\in \leb \rho \cap \leb \rho'$, see Lemma~\ref{dens1}(ii) and (iii), it follows that $x\in\leb G$.

Fix $\delta>0$ such that $\delta < d(x,y)+r$ and 
\begin{equation} \label{e:reg-delta}
 \frac{1}{2}\Leb(B(x,s)) \leq \Leb(G \cap B(x,s))
\end{equation}
for all $s \in (0,\delta)$ and fix $t>0$ such that $4t(d(x,y)+r)< \delta$.

We set $G_\delta:=G\cap B(x,\frac{\delta}{2})$, $A_\delta:=G_\delta \times B(y',r')$ and 
$\gamma_{\delta}:= \gamma \lfloor A_\delta$.
We shall prove that
\begin{equation} \label{e:main1}
 \frac{\rho'(x)}{4} \, \Leb(B(x,\frac{\delta}{2})) \leq (e_t \circ S)_\sharp \gamma_{\delta} (B(x,\delta)) 
\end{equation} 
and
\begin{equation} \label{e:main2}
 (e_t \circ S)_\sharp \gamma_{\delta} (B(x,\delta)) \leq 2^{2n+4} \rho(x) \Leb (T(\Gamma \cap [ B(x,\frac{\delta}{2}) \times B(y,r)]) \cap B(x,\delta)).
\end{equation}
Then \eqref{e:main1} and \eqref{e:main2} will yield 
\begin{equation*}
2^{-(2n+6)} \frac{\rho' (x)}{\rho(x)}\, \Leb(B(x,\frac{\delta}{2})) \leq \Leb(T(\Gamma \cap [ B(x,\frac{\delta}{2}) \times B(y,r)]) \cap B(x,\delta))
\end{equation*}
for any $\delta>0$ small enough which completes the proof.

To prove \eqref{e:main1}, we note that $(\pi_1)_\sharp \gamma_{\delta} \ll \Leb$ with density bounded below by $\frac{1}{2} \rho'(x) $ $\Leb$-a.e. on $G_\delta$. Together with \eqref{e:reg-delta}, it follows that
\begin{equation*}
 \frac{\rho'(x)}{4} \, \Leb(B(x,\frac{\delta}{2})) \leq (\pi_1)_\sharp \gamma_{\delta} (B(x,\frac{\delta}{2})).
\end{equation*}
Next, by choice of $\delta$ and $t$, we have $(e_t \circ S)(z,w) \in B(x,\delta)$ for all $z \in B(x,\frac{\delta}{2})$ and $w \in B(y,r)$, hence
\begin{equation*}
 B(x,\frac{\delta}{2}) \times B(y',r') \subset B(x,\frac{\delta}{2}) \times B(y,r) \subset (e_t \circ S)^{-1} (B(x,\delta))
\end{equation*}
and it follows that 
\begin{equation*}
 (\pi_1)_\sharp \gamma_{\delta} (B(x,\frac{\delta}{2})) = \gamma_{\delta} (B(x,\frac{\delta}{2}) \times B(y',r')) \leq (e_t \circ S)_\sharp \gamma_{\delta} (B(x,\delta))
\end{equation*}
and this completes the proof of \eqref{e:main1}.

We prove now \eqref{e:main2}. By hypothesis, $\gamma$ is a weak limit of solutions $\gamma_k$ to $(P_{\varepsilon_k})$ for some sequence $\e_k$ converging to 0. For each fixed $k\in\N$, we apply Lemma~\ref{restrictions} with $U = G_\delta \times \hn$ and Proposition~\ref{e:Linftydensityestimates} with $\overline \mu = (\pi_1)_\sharp (\gamma_k \lfloor U)$ and $\overline \nu = (\pi_2)_\sharp (\gamma_k \lfloor U)$. Taking into account the fact that $(\pi_1)_\sharp (\gamma_k \lfloor U) = \mu \lfloor G_\delta$, we get that $(e_t \circ S)_\sharp (\gamma_k \lfloor G_\delta \times \hn) \ll \Leb$ with density in $L^\infty$ and whose $L^\infty$-norm is bounded by 
\begin{equation} \label{e:rhok}
 \frac{1}{(1-t)^{2n+3}} \, \|\rho \lfloor_{G_\delta} \|_{L^\infty} \leq 2^{2n+4} \rho(x).
\end{equation}
Next we check that $(e_t \circ S)_\sharp (\gamma_k \lfloor G_\delta \times \hn)$ converges weakly to $(e_t \circ S)_\sharp (\gamma \lfloor G_\delta \times \hn)$. First it follows from Lemma \ref{weakres} (to be proved below) that $\gamma_k\lfloor G_\delta \times \hn$ converges weakly to $\gamma \lfloor G_\delta \times \hn$. Then, noting that $\gamma$ and each $\gamma_k$ are concentrated on $\Omega$ and that $e_t \circ S$ is continuous on $\Om$, the claim follows from Lemma \ref{conv} (to be proved below) applied with $\overline\gamma = \gamma \lfloor G_\delta \times \hn$, $\overline\gamma_k = \gamma_k \lfloor G_\delta \times \hn$, $B=\Om$ and $f = \varphi \circ e_t \circ S$ where $\varphi \in C_b(\hn)$. The fact that $\gamma$ is concentrated on $\Omega$ follows from Lemma \ref{pi1.1}. To check that $\gamma_k$ is concentrated on $\Omega$, denote by $\{y_i^k\}_i$ the finite set of the atoms of $(\pi_2)_\sharp \gamma_k$. We have that $\gamma_k$ is concentrated on $\hn\times \{y_i^k\}_i$. On the other hand  $\gamma_k(L_{ y_i^k} \times \{y_i^k\}) \leq \gamma_k(L_{ y_i^k} \times \hn) = \mu(L_{ y_i^k}) = 0$ since $\mu \ll \Leb$. It follows that $\gamma_k$ is concentrated on $\cup_i [(\hn \setminus L_{ y_i^k}) \times \{y_i^k\}] \subset \Om$. Then, taking into account \eqref{e:rhok}, we get
\begin{equation*}
 |\int_{\hn} \varphi \,\, d(e_t \circ S)_\sharp (\gamma \lfloor G_\delta \times \hn)| \leq 2^{2n+4} \rho(x) \, \|\varphi\|_{L^1}
\end{equation*}
for every $\varphi \in C_b(\hn)$. It follows that $(e_t \circ S)_\sharp (\gamma \lfloor G_\delta \times \hn)$ is in $(L^1)'$ with density in $L^\infty$ and whose $L^\infty$-norm is bounded by $2^{2n+4} \rho(x)$. Since $(e_t \circ S)_\sharp \gamma_\delta \leq (e_t \circ S)_\sharp (\gamma \lfloor G_\delta \times \hn)$, the same holds true for $(e_t \circ S)_\sharp \gamma_\delta$. Finally we note that $\gamma_\delta$ being concentrated on $\Gamma \cap [ B(x,\frac{\delta}{2}) \times B(y',r')] \subset \Gamma \cap [ B(x,\frac{\delta}{2}) \times B(y,r)]$, the measure $(e_t \circ S)_\sharp \gamma_\delta$ is concentrated on $T(\Gamma \cap [ B(x,\frac{\delta}{2}) \times B(y',r')]) \subset T(\Gamma \cap [ B(x,\frac{\delta}{2}) \times B(y,r)])$. All together we get
\begin{equation*}
\begin{split}
 (e_t \circ S)_\sharp \gamma_{\delta} (B(x,\delta)) &= (e_t \circ S)_\sharp \gamma_{\delta} (T(\Gamma \cap [ B(x,\frac{\delta}{2})
\times B(y,r)]) \cap B(x,\delta) )\\
&\leq  2^{2n+4} \rho(x) \Leb (T(\Gamma \cap [ B(x,\frac{\delta}{2}) \times B(y,r)]) \cap B(x,\delta))
\end{split}
\end{equation*}
which proves \eqref{e:main2}.
\end{proof}

\begin{lem} \label{weakres} 
Let $X$ be a separable and locally compact Hausdorff metric space in which every open set is $\sigma$-compact. Let $(\gamma_k)_k$ be a sequence in $\p (X\times X)$
which converges weakly to some $\gamma \in \p (X\times X)$ and
such that $(\pi_1)_\sharp \gamma_k = (\pi_1)_\sharp \gamma$ for every $k\in \N$. Then for any Borel set $G \subset X$, the sequence 
$(\gamma_k \lfloor  G\times X)_k$ converges weakly to $\gamma \lfloor G\times X$.
\end{lem}

\begin{proof} We have to prove that for any $\varphi \in C_b(X)$,
\begin{equation*}
\lim_{k\rightarrow +\infty} \int_{X\times X} \1_G(x) \varphi (x,y) \,d \gamma_k(x,y) = \int_{X \times X} \1_G(x) \varphi(x,y) \,d \gamma(x,y).
\end{equation*}

It follows from Lusin's Theorem that for any 
$\e>0$ there exists a closed set $F_\e$ such that ${\1_G} \lfloor_{F_\e}$ is continuous and $(\pi_1)_\sharp \gamma(X \setminus
F_\e) < \e$. As a consequence, for every $\e >0$, the restriction of $(x,y) \mapsto \1_G(x) \varphi (x,y)$ to $F_\e \times X$ is continuous and 
$$\limsup_{k\rightarrow +\infty} \gamma_k ((X\setminus F_\e) \times X)
= (\pi_1)_\sharp \gamma (X \setminus F_\e) < \e.$$
Then since $(x,y)\mapsto |\1_G(x) \varphi (x,y)|$ is
bounded and hence uniformly integrable with respect to
$(\gamma_k)_k$, the claim follows from \cite[Proposition 5.1.10]{ags}.
\end{proof}

\begin{lem} \label{conv}
Let $X$ be a separable metric space and $(\overline\gamma_k)_k$ be a sequence in $\p(X)$ which converges weakly 
to some $\overline\gamma \in\p(X)$. Let $f:X\to \R$ be a measurable and bounded function which is continuous 
in $B$ for some Borel set $B\subset X$ such that $\overline\gamma_k(X\setminus B)=0$ for every $k\in \N$ and $\overline\gamma(X\setminus B)=0$, then
\begin{equation*}
 \lim_{k \to \infty} \int_X f d\gamma_k= \int_X f d\gamma.
\end{equation*}
\end{lem}

\begin{proof} Let $\overline{f}$ and $\tilde{f}$ be respectively the lower and
 upper semicontinuous envelope of $f$. We have $\overline{f}=f=\tilde{f}$ on $B$ and hence $\gamma$-a.e. and $\gamma_k$-a.e. for every $k\in\N$. It follows that
\begin{multline*}
 \int_X f \, d\gamma = \int_X \overline f \, d\gamma \leq \liminf_{k \to \infty} \int_X \overline{f} \,d\gamma_k 
= \liminf_{k \to \infty} \int_X f  \,d\gamma_k \\
\leq \limsup_{k \to \infty} \int_X f d\gamma_k 
= \limsup_{k \to \infty} \int_X \tilde f d\gamma_k \leq \int_X \tilde f d\gamma = \int_X f \, d\gamma
\end{multline*}
which proves the claim.
\end{proof}

%%%%
%%%%%%%%%%%%%%%%%%%%%%%
\section{Solution to Monge's problem} \label{conclusion}

We prove that optimal transport plans in $\Pi_1(\mu,\nu)$ that are obtained as weak limit of solutions of the variational approximations introduced in Section~\ref{varapprox} are induced by a transport, hence giving a solution to Monge's transport problem as stated in Theorem~\ref{mainthm}. Note that due to the fact that $\Pi$ is relatively compact in $\p(\hn\times\hn)$, such optimal transport plans do exist.

\begin{thm} \label{mainthmbis}
 Let $\e_k$ be a sequence converging to 0 and $\g_{\e_k}$ a sequence of solutions to $(P_{\e_k})$ which is weakly converging to some $\g\in \p(\hn\times\hn)$. Then $\g$ is concentrated on a $\mu$-measurable graph and hence induced by a transport.
\end{thm}

\begin{proof}
First we know from Lemma~\ref{optpi2} that $\g\in \Pi_2(\mu,\nu)$. From the previous sections and using inner regularity of Borel probability measures, one can then find $\s$-compact sets $\G$ and $\G'$ such that $\G'\subset \G \subset \Omega$ and the conclusions of Lemma~\ref{pi2}, Lemma~\ref{dens2} and Lemma~\ref{mainlemma} hold. We prove here that for any $x\in \pi_1(\G')$ there is a unique $y\in \hn$ such that $(x,y)\in \G'$.

By contradiction, assume that one can find $x_0 \in \pi_1(\G')$ and $(x_0,y_0)\in \G$, $(x_0,y_1)\in \G$ with $y_0\not=y_1$. Without loss of generality one can assume that $d(x_0,y_0) \leq d(x_0,y_1)$ and $x_0\not=y_1$. Then, by Lemma \ref{dens2} and Lemma \ref{mainlemma}, for all $r>0$ and for all $\delta>0$ small enough, one can find $x'\in B(x_0,\delta) \cap \G^{-1} (B(y_0,r)) \cap T(\G \cap [B(x_0,\frac{\delta}{2}) \times B(y_1,r)])$. It follows that one can find $y' \in B(y_0,r)$ such that $(x',y') \in \G$ and $(x,y) \in \G \cap (B(x_0,\frac{\delta}{2})\times B(y_1,r))$ such that $x\not=y$, $x'\not=x$ and $x'$ lie on the minimal curve between $x$ and $y$. Then it follows from Lemma~\ref{pi2} that $x$, $x'$, $y$ and $y'$ lie on the same minimal curve ordered in that way.

Assume first that $ d(x_0,y_0) < d(x_0,y_1)$. We know from Lemma \ref{pi2} that $d(x,y)\leq d(x,y')$. On the other hand, we have
\begin{equation*}
\begin{split}
 d(x,y')&\leq d(x,x_0) + d(x_0,y_0) + d(y_0,y')\\
& \leq d(x_0,y_0) + \dfrac{\delta}{2} + r \\
&= d(x_0,y_1) + d(x_0,y_0) - d(x_0,y_1) + \dfrac{\delta}{2} + r \\
&\leq  d(x_0,x) + d(x,y) + d(y,y_1) + d(x_0,y_0) - d(x_0,y_1) + \dfrac{\delta}{2} + r\\
&\leq d(x,y) + d(x_0,y_0) - d(x_0,y_1) + \delta + 2r.
\end{split}
\end{equation*}
It follows that $d(x,y')< d(x,y)$ provided we take $r>0$ and $\delta>0$ small enough which gives a contradiction. If $ d(x_0,y_0) = d(x_0,y_1)$, we have
\begin{equation*}
\begin{split}
d(x_0,y_1) &\leq d(x_0,x) + d(x,y) + d(y,y_1) \\
& = d(x_0,x) + d(x,y') - d(y',y) +  d(y,y_1)\\
& \leq d(x,y') - d(y',y) + \dfrac{\delta}{2} + r ,
\end{split}
\end{equation*}
\begin{equation*}
 d(x,y') \leq d(x,x_0) + d(x_0,y_0) + d(y_0,y') \leq d(x_0,y_0) + \dfrac{\delta}{2} + r,
\end{equation*}
\begin{equation*}
 d(y',y) \geq d(y_0,y_1) - d(y_0,y') - d(y_1,y)\geq d(y_0,y_1) -2r,
\end{equation*}
hence,
\begin{equation*}
 d(x_0,y_1) \leq d(x_0,y_0) - d(y_0,y_1) + 4r +\delta. 
\end{equation*}
It follows that $d(x_0,y_1) < d(x_0,y_0)$ provided we take $r>0$ and $\delta>0$ small enough which gives also a contradiction.
\end{proof}

%%%%%%
%%%%%%%%%%%%
\section{Extension to more general metric measure spaces} \label{extensions}

First we note that a major part of intermediate steps in the strategy adopted in the present paper can be naturally extended to Polish and non-branching geodesic spaces equipped with a reference measure for which the Lebesgue's differentiation theorem holds. 

Next our choice of approximating costs $c_\e$ in the approximation procedure is not the only possible one. This choice could in particular be adapted to fit other contexts (for instance concerning the relevant properties of solutions to the transport problem associated to the approximating cost).

Finally the Measure Contraction Property is here technically very convenient. We note however that this property is unnecessarily too strong for what is actually needed in the proof about the lower density of the transport set. Much local and weaker versions about the behavior of the measure of sets transported along minimal curves are indeed sufficient as clearly shows up from the proof.

This approach can in particular be adapted to give an alternative proof of the existence of solutions to Monge's transport problem in the Riemannian setting without using Sudakov's type arguments.

For the reasons listed above it is furthermore very likely that the present strategy could be adapted and extended to other geodesic metric spaces.

%%%%
%%%%%%%%%%%%%%%%%%%%%%%


\begin{thebibliography}{[99]}

\bibitem{ambrosio} L.~Ambrosio, \textit{Lecture notes on optimal transport problems.} Mathematical aspects of evolving interfaces (Funchal, 2000), 1--52, Lecture Notes in Math., 1812, Springer, Berlin, 2003. 

\bibitem{ags} L.~Ambrosio, N.~Gigli, G.~Savar\'e, \textit{Gradient flows in metric spaces and in the space of probability measures.} Lectures in Mathematics ETH Z\"urich. Birkh\"auser Verlag, Basel, 2005. viii+333 pp.

\bibitem{akp} L.~Ambrosio, B.~Kirchheim, A.~Pratelli, \textit{Existence of optimal transport maps for crystalline norms.} Duke Math. J. 125 (2004), no. 2, 207--241. 

\bibitem{ap} L.~Ambrosio, A.~Pratelli, \textit{Existence and stability results in the $L\sp 1$ theory of optimal transportation.}  Optimal transportation and applications (Martina Franca, 2001),  123--160, Lecture Notes in Math., 1813, Springer, Berlin, 2003.

\bibitem{ar} L.~Ambrosio, S.~Rigot, \textit{Optimal mass transportation in the Heisenberg group.} J. Funct. Anal. 208 (2004), no. 2, 261--301. 

\bibitem{ab} G.~Anzellotti, S.~Baldo, \textit{Asymptotic development by $\Gamma$-convergence.}  Appl. Math. Optim.  27  (1993),  no. 2, 105--123.

\bibitem{bianc-cav} S.~Bianchini, F.~Cavalletti, \textit{The Monge Problem For Distance Cost In Geodesic Spaces}, Preprint.

\bibitem{bel} Sub-Riemannian geometry. Edited by André Bella\"iche and Jean-Jacques Risler. Progress in Mathematics, 144. Birkh\"auser Verlag, Basel, 1996.

\bibitem{caravenna1} L.~Caravenna, \textit{A proof of Sudakov theorem with strictly convex norms.} Preprint.

\bibitem{cfm} L.~Caffarelli, M.~Feldman, R.J.~McCann, \textit{Constructing optimal maps for Monge's transport problem as a limit of strictly convex costs.} J. Amer. Math. Soc. 15 (2002), no. 1, 1--26.

\bibitem{champion-dePascale-first} T.~Champion, L.~De Pascale, \textit{The Monge problem for strictly convex norm in $\R^d$}, to appear on  J. Eur. Math. Soc. 

\bibitem{champion-dePascale} T.~Champion, L.~De Pascale, \textit{The Monge problem in $\R^d$.} Preprint.

\bibitem{evans-gangbo} L. C.~Evans, W.~Gangbo, \textit{Differential equations methods for the Monge-Kantorovich mass transfer problem.} Mem. Amer. Math. Soc. 137 (1999), no. 653, viii+66 pp. 

\bibitem{FeldMcC} M.~Feldman, R.J.~McCann, \textit{Monge's transport problem on a Riemannian manifold.} Trans. Amer. Math. Soc. 354 (2002), no. 4, 1667--1697.

\bibitem{figjuil} A.~Figalli, N.~Juillet, \textit{Absolute continuity of Wasserstein geodesics in the Heisenberg group.}  J. Funct. Anal.  255  (2008),  no. 1, 133--141.

\bibitem{gaveau} B.~Gaveau, \textit{Principe de moindre action, propagation de la chaleur et estimées sous elliptiques sur certains groupes nilpotents.}  Acta Math.  139  (1977), no. 1-2, 95--153. 

\bibitem{heinonen} J.~Heinonen, \textit{Lectures on analysis on metric spaces.} Universitext. Springer-Verlag, New York, 2001. x+140 pp.

\bibitem{juillet} N.~Juillet, \textit{Geometric Inequalities and Generalized Ricci Bounds in the Heisenberg Group.} Int Math Res Notices, 13 (2009), 2347--2373

\bibitem{mont} R.~Montgomery, \textit{A tour of subriemannian geometries, their geodesics and applications.} Mathematical Surveys and Monographs, 91. American Mathematical Society, Providence, RI, 2002. xx+259 pp. 

\bibitem{monti} R.~Monti, \textit{Some properties of Carnot-Carathéodory balls in the Heisenberg group.}  Atti Accad. Naz. Lincei Cl. Sci. Fis. Mat. Natur. Rend. Lincei (9) Mat. Appl.  11  (2000),  no. 3, 155--167 (2001).

\bibitem{pansu} P.~Pansu, \textit{M\'etriques de Carnot-Carath\'eodory et quasiisom\'etries des espaces sym\'etriques de rang un.} Ann. of Math. (2)  129  (1989),  no. 1, 1--60. 

\bibitem{sant} F.~Santambrogio, \textit{Absolute continuity and summability of transport densities: simpler proofs and new estimates.} preprint 2009.

\bibitem{sudakov} V. N.~Sudakov, \textit{Geometric problems in the theory of infinite-dimensional probability distributions.} Proc. Steklov Inst. Math. 1979, no. 2, i--v, 1--178. 

\bibitem{TrudWang} N. S.~Trudinger, X.J.~Wang, \textit{On the Monge mass transfer problem.} Calc. Var. Partial Differential Equations 13 (2001), no. 1, 19--31. 

\bibitem{villani} C.~Villani, \textit{Optimal transport. Old and new.} Grundlehren der Mathematischen Wissenschaften, 338. Springer-Verlag, Berlin, 2009. xxii+973 pp.

\end{thebibliography}
\end{document}